\documentclass[onefignum,onetabnum]{siamart220329}

\usepackage{enumerate}
\usepackage{enumitem}
\usepackage{amsmath}
\usepackage{amssymb}
\usepackage{latexsym,bm,amsfonts}

\newtheorem{thm}{Theorem}
\newtheorem{lem}{Lemma}
\newtheorem{prop}{Proposition}
\newtheorem{coro}{Corollary}

\newtheorem{defi}{Definition}
\newtheorem{asp}{Assumption}

\usepackage{lipsum}
\usepackage{amsfonts}
\usepackage{graphicx}
\usepackage{epstopdf}
\usepackage{algorithmic}
\usepackage{hyperref}
\ifpdf
  \DeclareGraphicsExtensions{.eps,.pdf,.png,.jpg}
\else
  \DeclareGraphicsExtensions{.eps}
\fi

\newsiamremark{remark}{Remark}
\newsiamremark{hypothesis}{Hypothesis}
\crefname{hypothesis}{Hypothesis}{Hypotheses}
\newsiamthm{claim}{Claim}

\headers{ }{J. Zhang, W. Pu, and Z.-Q. Luo.}

\title{On the Iteration Complexity of Smoothed Proximal ALM for Nonconvex Optimization Problem with Convex Constraints\thanks{Submitted to the editors DATE. Jiawei Zhang and Wenqiang Pu contributed equally.
\funding{The work of Wenqiang Pu is supported by the National Natural Science Foundation of China (No. 62101350). The work of Zhi-Quan Luo is supported by the National Natural Science Foundation of China (No. 61731018) and the Guangdong Provincial Key Laboratory of Big Data Computing.}}}

\author{Jiawei Zhang$^\ddagger$\thanks{Laboratory for Information and  Decision Systems, Massachusetts Institute of Technology, USA
  (\email{jwzhang@mit.edu}).}
\and Wenqiang Pu\thanks{Shenzhen Research Institute of Big Data, The Chinese University of Hong Kong, Shenzhen, China 
  (\email{wenqiangpu@cuhk.edu.cn}, \email{luozq@cuhk.edu.cn}).}
\and Zhi-Quan Luo\footnotemark[3]}

\usepackage{amsopn}

\ifpdf
\hypersetup{
  pdftitle={ },
  pdfauthor={J. Zhang, W. Pu, and Z.-Q. Luo.}
}
\fi

\begin{document}

\maketitle

\begin{abstract}
It is well-known that the lower bound of iteration complexity for solving nonconvex unconstrained optimization problems is $\Omega(1/\epsilon^2)$, which can be achieved by standard gradient descent algorithm when the objective function is smooth. This lower bound still holds for nonconvex constrained problems, while it is still unknown whether a first-order method can achieve this lower bound. In this paper, we show that a simple single-loop first-order algorithm called smoothed proximal augmented Lagrangian method (ALM) can achieve such iteration complexity lower bound. The key technical contribution is a strong local error bound for a general convex constrained problem, which is of independent interest.
\end{abstract}



\section{Introduction}
Consider the following optimization problem with nonlinear constraints:
\begin{equation}\label{P1}
\begin{aligned}
&\mbox{minimize}& \ &f(x)\\ 
&\mbox{subject to}& \ &Ax=b,\ x \in\mathcal{X},
\end{aligned}
\end{equation}
where matrix $A\in \mathbb{R}^{m\times n}$, $b\in\mathbb{R}^m$, $\mathcal{X}=\{x\mid h_i(x)\leq 0,i=m+1,\ldots, m+\ell\}$ is a convex set, $h_i(\cdot):\mathbb{R}^n\mapsto\mathbb{R}$ is a smooth convex function. The objective function $f(\cdot):\mathbb{R}^n\mapsto\mathbb{R}$ is assumed to be smooth but possibly nonconvex, whose gradient is Lipschitz-continuous. Problem~\eqref{P1} appears in many practical applications,
such as principal component analysis~\cite{wang2021linear}, resource allocation~\cite{yan2020collaborative}, matrix separation~\cite{matrixsep}, phase retrieval~\cite{phase}, community detection~\cite{wang2021non}, distributionally robust learning~\cite{robustleanring}, image background/foreground extraction~\cite{imageback}, just to mention a few.

For large problem size $n$, it is popular to consider first-order algorithms to solve problem~\eqref{P1}. When $A,b$ are zeros and $\mathcal{X}=\mathbb{R}^n$, problem~\eqref{P1} becomes an unconstrained optimization problem. It was shown in~\cite{carmon2020lower} that $\Omega(1/\epsilon^2)$ iterations are needed for any first-order algorithms to attain an $\epsilon$-solution ($\|\nabla f(x) \|\leq \epsilon$) of this unconstrained problem. The gradient descent algorithm can find an $\epsilon$-solution in $\mathcal{O}(1/\epsilon^2)$ iterations~\cite{Bertsekas}. Therefore, gradient descent algorithm is \textit{order-optimal} for unconstrained problems. The iteration complexity lower bound $\Omega(1/\epsilon^2)$ still holds~\cite{zhang2020global} for constrained problem~\eqref{P1}. However, whether there is an \textit{order-optimal} first-order algorithm ($\mathcal{O}(1/\epsilon^2)$ iteration complexity) for it has remained as an open question in the literature. 

\subsection{Existing Algorithms for Problem~\eqref{P1}}
\subsubsection{ALM and ADMM}
To deal with constrained optimization problems, the augmented Lagrangian method (ALM) is often a preferred method. ALM solves problem~\eqref{P1} by dualizing and penalizing the equality constraint and searches the saddle points of the resulting augmented Lagrangian function. 
Moreover, in many practical applications, the optimization variable can be divided into variable blocks. In this case, the alternating direction method of multipliers (ADMM) is widely used. The ADMM alternately performs exact or inexact minimization steps to the primal variable blocks and then updates the dual variable by a dual ascent step. Convergence analysis for ALM and ADMM has been studied in the past decades and hereafter we briefly review existing convergence results.

\begin{itemize}[fullwidth,itemindent=0em,label=$\bullet$]
\item \textbf{Convex Case:} ALM is known to converge for convex problems under mild conditions~\cite{Luo-Hong12} and the convergence of ADMM with two blocks is shown in~\cite{eckstein-bertsekas}. The authors of~\cite{Deng2016} prove the linear convergence of ADMM for two-block case ($k=2$) under the assumption that one block is strongly convex and smooth. For multi-block case ($k\ge 3$), the authors of~\cite{paral} study the convergence of ADMM for the consensus problem. In~\cite{He2012}, a variant of ADMM for convex smooth objective functions is studied. The linear convergence of multi-block ADMM for a family of convex non-smooth problems is established in~\cite{Luo-Hong12} by using error bound analysis. In short, convergence of ALM and ADMM is well understood under convex setting.

\item \textbf{Nonconvex Case:} The local convergence analysis of the ALM for problems with a smooth objective function and smooth equality constraint is given in~\cite{Bertsekas}. However, the global convergence of both ALM and ADMM under nonconvex setting are studied only recently. The authors in~\cite{Hong17} prove that proximal ALM converges for a class of nonconvex problems where part of the optimization variable needs to be unconstrained. In~\cite{Hong-Luo16}, the convergence of ADMM is established for consensus-based sharing problems by using the augmented Lagrangian function as the potential function. This approach is further extended in~\cite{wyin16} to a large family of nonconvex-nonsmooth problems. The papers~\cite{JiangMa16, GLI15} prove the convergence of inexact ADMM for certain nonconvex and nonsmoooth problems. These references all require \textit{at least one block of the variable to be unconstrained}. Namely, for the linear equality constraint $Ax_1+Bx_2=b$, the image of $A$ is a subset of the image of $B$ and $x_2$ does not have any other constraint. Another work in~\cite{Gao-Goldfarb-Curtis-2018} establishes the convergence of multi-block ADMM algorithm for the so called multi-affine constraints which are linear in each variable block but otherwise nonconvex. This work also requires some technical assumptions, including the similar feasibility assumption and the objective function for some block must be strongly convex. In the papers mentioned above, they can achieve $\mathcal{O}(1/\epsilon^2)$ iteration complexity. But these results do not apply to problem~\eqref{P1}.
\end{itemize}

\subsubsection{Penalty Method}
Besides, the penalty method is another popular choice for dealing equality constaint. The penalty methods proposed in~\cite{Monteiro19,lin2022complexity} can be used to solve problem \eqref{P1}. These algorithms solve~\eqref{P1} by penalizing the linear equality constraints and solving the penalized objective function using inexact proximal method, where the strongly convex sub-problems are approximately solved by accelerated gradient method. The best iteration complexity for solving problem~\eqref{P1} is $\mathcal{O}(1/\epsilon^{2.5})$~\cite{lin2022complexity}, which does not match the lower iteration complexity bound $\Omega(1/\epsilon^2)$ for nonconvex problems. It is interesting to investigate whether there is an algorithm can achieve such iteration complexity lower bound.

\subsubsection{Smoothed Proximal ALM Method}
Recently, authors in~\cite{zhang2018proximal} propose a smoothed proximal primal-dual  method for solving \eqref{P1} with $\mathcal{X}$ being a bounded box and prove the $\mathcal{O}(1/\epsilon^2)$ iteration complexity under some regularity conditions.
The results in~\cite{zhang2018proximal} are extended to general linear constrained cases~\cite{zhang2020global}. However, both~\cite{zhang2018proximal} and~\cite{zhang2020global} only focus on linearly constrained problems. The work~\cite{zeng2021moreau} proves an $\mathcal{O}(1/\epsilon^2)$ iteration complexity for more general problems based on a new potential function. Its proof is based on a strong assumption, that the constant of a certain error bound (similar to the dual error bound of~\cite{zhang2018proximal,zhang2020global}) is smaller than a certain threshold. The smoothing technique in~\cite{zhang2018proximal} has recently been extended for difference-of-convex programs~\cite{sun2021algorithms} with special linear constraints. Recently, authors in~\cite{kong2020iteration} establish an $\mathcal{O}(1/\epsilon^{3})$ complexity bound of the proximal ALM method for nonconvex composite optimization with nonlinear convex constraints.

\subsection{Error Bound Analysis}
To answer whether there is an order-optimal first-order algorithm for problem \eqref{P1}, it is necessary to derive upper bounds for some primal and dual errors. This leads to the so-called \textit{error bound analysis}, which is the key for the convergence analysis. 

The error bound analysis, which uses the optimality residual to bound the distance from the current iterate to the solution set of a given optimization problem, is studied for a long time in the optimization literature~\cite{eb-survey,Pang-survey} for first-order algorithms. Papers such as~\cite{So19} also use this error bound to analyze second-order optimization methods. A fundamental error bound is given by Hoffman~\cite{Hoffmanbound}, which uses the violation of constraints to bound the distance from a point to a polyhedral set. Hoffman bound has been extended to more general nonlinear systems in~\cite{luo-Hoffman,luo-analytic}.

For optimization problems, two types of error bounds--the primal error bound and dual error bound, are considered in the literature. The primal error bound uses some primal optimization residuals such as the gradient norm and proximal gradient norm to bound the distance to the solution set while the dual error bound uses the dual residuals such as the constraint residual to bound the distance. For instance, the author of~\cite{pang1987posteriori} proves a global primal error bound for strongly convex optimization problems and authors of~\cite{luo-linear} prove global error bounds for some special types of monotone linear complementarity problems and convex quadratic problems. Papers~\cite{composite,bcd} establish some local primal error bound when the strong convexity is absent. Recently, a new framework which establishes primal error bounds for a class of structured convex optimization problems is proposed in~\cite{zhou2017unified}. For dual error bound, authors of~\cite{luo-dual} prove a `local' dual error bound and use it to prove the linear convergence of the dual ascent algorithm for a family of convex problems (without strong convexity) with polyhedral constraints. Authors of~\cite{Luo-Hong12} show the linear convergence of ADMM algorithm combining the primal and dual error bounds, again without strong convexity.
In~\cite{zhang2018proximal,zhang2020global}, the authors derive local and global dual error bounds for problems with linear constraints and use them to show the convergence of smoothed ALM for solving linearly constrained problems.
These dual error bounds only hold for linearly constrained problems and cannot be extended to nonlinearly constrained problems.
In the literature, papers~\cite{svm,gebkkt} consider nonlinearly constrained problems and prove some error bounds using KKT residual rather than the constraint residual, to bound the distance to the solution set. This type of error bounds~\cite{svm,gebkkt} are nonhomogeneous, i.e., the exponent of the two sides of the inequality are not the same. This implies that the error bounds are weak near the solution set.

\subsection{Our Contribution}
In this paper, we prove a homogeneous error bound for strongly convex, nonlinearly  constrained problems and use it to analyze the convergence rate of the smoothed proximal ALM for problem~\eqref{P1}. In particular, we prove a general dual error bound for a regularized version of nonconvex problem with nonlinear constraints. Equipped with this error bound, we prove that the smoothed proximal ALM algorithm can find an $\epsilon$-solution with an $\mathcal{O}(1/\epsilon^2)$ iteration complexity under some regularity assumptions. To the best of our knowledge, this is the first result that guarantees the $\mathcal{O}(1/\epsilon^2)$ iteration complexity of first-order algorithms for problem~\eqref{P1} and meets the lower iteration complexity bound.

\section{Main Result}
\subsection{Notations}
Notations frequently used are listed below, others will be explained when they firstly appear.
\begin{itemize}[fullwidth,itemindent=0em,label=$\bullet$]
\setlength{\parskip}{0pt}
\item Index set $\{1, 2, \cdots, m\}$ is denoted as $[m]$ and $\{m, m+1, \cdots, n\}$ is denoted as $[m: n]$.  
\item The smallest (largest) singular value of a matrix $A$ is denoted as $\sigma_{\rm{min}}(A)$ ($\sigma_{\rm{max}}(A)$).
\item Denote $J(x)\in\mathbb{R}^{n\times \ell}$ as the Jaccobian matrix associating with $(h_1(x), \cdots, h_\ell(x))$.
\item For a vector  $v\in\mathbb{R}^{(m+\ell)}$ and an index set $\mathcal{S}\subset[m+\ell]$, $v_{\mathcal{S}}\in \mathbb{R}^{|\mathcal{S}|}$ means the vector that consists of all coordinates of $v$ belonging to $\mathcal{S}$.
\item For a matrix $Q\in\mathbb{R}^{(m+\ell)\times n}$ and an index set $\mathcal{S}\subseteq [m+\ell]$, $Q_\mathcal{S}\in\mathbb{R}^{|\mathcal{S}|\times n}$ means the row submatrix of $Q$ corresponding to the indices in $\mathcal{S}$.
\end{itemize}

\subsection{Stationary Solution Set of Problem~\eqref{P1}}\label{esolution}

Let $\iota(x)$ denote the indicator function of the constraint set $\mathcal{X}$, i.e., $\iota(x)=0$ if $x\in \mathcal{X}$ and $\iota(x)=\infty$ otherwise. We define the notion of $\epsilon$-stationary solution of problem~\eqref{P1} as follows: 

\begin{defi}[$\epsilon$-Stationary Solution]\label{def:soluset}
A vector $x\in\mathcal{X}\subseteq\mathbb{R}^n$ is said to be an $\epsilon$-stationary solution ($\epsilon\geq 0$) of problem~\eqref{P1} if there exist $y\in\mathbb{R}^m$ and $v\in \nabla f(x)+A^Ty+\partial{\iota(x)}$ such that $\|v\|\le \epsilon$  and $\|Ax-b\|\le \epsilon$, where $\partial{\iota(x)}$ is the sub-differential set of $\iota(\cdot)$ at $x$. Further, the stationary solution set $\mathcal{X}^*$ is defined as the set of all $0$-stationary solutions.
\end{defi}

Though a stationary solution is not necessarily a global minimizer, in practice with a good initialization, the stationary solutions obtained by the first-order algorithms are usually of good quality. Therefore, we focus on finding an $\epsilon$-stationary solution of problem~\eqref{P1}.

\subsection{The Smoothed Proximal ALM}
We first define the augmented Lagrangian function of problem~\eqref{P1} as:
$$L(x; y)=f(x)+y^T(Ax-b)+\frac{\gamma}{2}\|Ax-b\|^2,$$
where $\gamma\geq 0$ is a parameter which penalizes the equality violation. At iteration $t$, the well-konwn augmented Lagrangian multiplier method (ALM) update is as follows:
\begin{equation*}
\textrm{ALM:}\  
\left\{ 
    \begin{aligned}
    y^{t+1}&=y^t+\alpha(Ax^t-b)\\
    x^{t+1}&=\arg\min_{x\in \mathcal{X}}\ L(x; y^{t+1})
    \end{aligned}
\right.,
\end{equation*}
where $\alpha>0$ is the dual step size. By linearizing $L(x; y^{t+1})$, the proximal ALM (Prox-ALM) has the following updating rule:
\begin{equation*}
\textrm{Prox-ALM:}\  
\left\{ 
\begin{aligned}
    y^{t+1}&=y^t+\alpha(Ax^t-b)\\
    x^{t+1}&=\arg\min_{x\in \mathcal{X}}\ \{\langle\nabla_xL(x^t;y^{t+1}),x-x^t \rangle+\frac{p}{2}\| x- x^{t}\|^2\}
    \end{aligned}
    \right.,
\end{equation*}
where $p>0$ is the proximal parameter. The difference between ALM and Prox-ALM is the subproblem for $x$. The ALM directly minimizes $L(x;y)$ over $\mathcal{X}$ which may be hard to solve since $L(x;y)$ may not be convex in $x$. While Prox-ALM minimizes a special (strongly convex) quadratic function, which is equivalent to a projection problem associating with convex set $\mathcal{X}$, i.e., $x^{t+1}=\mathcal{P}_{\mathcal{X}}(x^t-\frac{1}{p}\nabla_xL(x^t;y^{t+1}))$, where $\mathcal{P}_{\mathcal{X}}$ is the projection operator defined as $\mathcal{P}_{\mathcal{X}}(\bar{x})=\arg\min_{x\in\mathcal{X}}\ \|x - \bar{x} \|^2.$ For many convex sets, the projection operator $\mathcal{P}_{\mathcal{X}}(\cdot)$ can be computed efficiently or even with analytic forms~\cite{parikh2014proximal}.

Both ALM and Prox-ALM are convergent under some mild assumptions~\cite{Luo-Hong12} if $f$ is convex. However, the counter-example in~\cite{wyin16} shows ALM may oscillate when $f$ is nonconvex. A numerical example in~\cite{zhang2018proximal} shows that Prox-ALM can oscillate for a box constrained nonconvex quadratic problem. The convergence of Prox-ALM for problem~\eqref{P1} can not be guaranteed in general. 

Recently, a smoothed Prox-ALM (see \cref{Alg:SProxALM}) with $\mathcal{O}(1/\epsilon^{2})$ iteration complexity is proposed in~\cite{zhang2018proximal,zhang2020global}, where $\mathcal{X}$ is considered to be box constraint~\cite{zhang2018proximal} or polyhedron~\cite{zhang2020global}. The smoothed Prox-ALM is a primal-dual algorithm, whose primal update is based on the following function: 
\begin{equation}\label{eq:K}
K(x, z;y)=f(x)+y^T(Ax-b)+\frac{\gamma}{2}\|Ax-b\|^2+\frac{p}{2}\|x-z\|^2,\end{equation}
where $p>0$ is a constant. The convergence analysis in~\cite{zhang2018proximal,zhang2020global} relies on Hoffman bound over polyhedron, which can not be used for general convex constraint $\mathcal{X}$. In the next, we first give basic assumptions and then present our main convergence result of smoothed Prox-ALM for problem~\eqref{P1}.
\begin{algorithm}
	\caption{Smoothed Proximal ALM}
	\label{Alg:SProxALM}
\begin{algorithmic}
\STATE{Let $\alpha>0$, $0<\beta\le 1$, and $c>0$;}
\STATE{Initialize $x^0\in\mathcal{X}, z^0\in\mathcal{X}, y^0\in\mathbb{R}^m$;}
\FOR{$t=0,1,2,\ldots,$}
\STATE{$y^{t+1}=y^t+\alpha(Ax^t-b)$;}
\STATE{$x^{t+1}=\mathcal{P}_{\mathcal{X}}(x^t-c\nabla_x K(x^t, z^t; y^{t+1}))$;}
\STATE{$z^{t+1}=z^t+\beta(x^{t+1}-z^t)$.}
\ENDFOR
\end{algorithmic}
\end{algorithm}

\subsection{Assumptions}
We state our main assumptions, which are valid in many practical problems.

\begin{asp}[Basic Assumptions]\label{ass:basic}
$\ $
\begin{enumerate}
    \item $\mathcal{X}$ is a compact convex set.
    \item \label{ass:bounded}$f(x)$ is bounded from below in $\mathcal{X}$, i.e., $f(x)> \underline{f}>-\infty,\ \forall x\in\mathcal{X}$.
    \item \label{ass:lip}$f(x)$ is a smooth function with $L_f$-Lipschitz-continuous gradient in $\mathcal{X}$, i.e., 
    $$
    \|\nabla f(x) -  \nabla f(x^\prime)\| \leq L_f\| x-x^\prime\|,\ \forall x,x^\prime\in\mathcal{X}.
    $$
    \item \label{ass:lip_h}Functions $h_i(x),\forall i$, are smooth convex with $L_h$-Lipschitz-continuous gradient.
    \item \label{ass:slater}The Slater condition holds for problem~\eqref{P1}.
\end{enumerate}
\end{asp}
\cref{ass:basic}.\ref{ass:bounded} ensures problem~\eqref{P1} is well-defined and \cref{ass:basic}.\ref{ass:lip} is a common assumption for convergence analysis. By \cref{ass:basic}.\ref{ass:lip}, we know that $K(x,z;y)$ defined in~\eqref{eq:K} is $(p-L_f)$-strongly convex of $x$ and its gradient $\nabla_x K(x, z;y)$ is $(L_f+\gamma\sigma_{{\textrm{max}}}^2(A)+p)$-Lipschitz continuous. The Slater condition in \cref{ass:basic}.\ref{ass:slater} is a basic assumption for problems with convex constraints. According to the Karush--Kuh--Tucker (KKT) Theorem~\cite{Bertsekas}, the Slater condition implies any vector $x^*\in\mathbb{R}^n$ is a stationary solution (see \cref{def:soluset}) if and only if it satisfies the KKT conditions of problem~\eqref{P1}, given as below:

\begin{equation}\label{eq:kkt_ori}
\begin{aligned}
\nabla f(x^*)+A^Ty^*+J^T(x^*)\mu^*&=0,\\
Ax^*-b&=0,\\
h_i(x^*)&\le 0,\   i\in [m+1,  m+\ell],\\
\mu_i^* h_i(x^*)=0,\mu_i^*&\ge 0,\   i\in [m+1,  m+\ell],
\end{aligned}
\end{equation}
where $y^*\in\mathbb{R}^m$  and $\mu^*\in\mathbb{R}^\ell$ are Lagrangian multipliers (dual variables) associated with the equality and inequality constraints respectively, $\mu_i^*$ denotes the $i$th component of $\mu^*$, and $J(x^*)$ is the Jaccobian matrix evaluated at $x^*$. Moreover, by \cref{ass:basic}.\ref{ass:slater}, the Slater condition also holds if the equality constraints are slightly perturbed. Specifically, we have the following proposition:
\begin{prop}\label{prop:slater}
There exists a constant $\Delta>0$ such that 
the Slater condition holds for the set $\{x\mid Ax=b+r, x\in \mathcal{X}\}$ with any $r\in\textrm{Range}(A)$ satisfying $\|r\|\le \Delta$.
\end{prop}

\begin{proof}
By \cref{ass:basic}.\ref{ass:slater}, we have $x\in\mathcal{X}$ satisfying $h_i(x)<0,i\in[m+1,m+\ell]$ and $Ax=b$. According to \cref{ass:basic}.\ref{ass:lip_h}, there exists a sufficiently small $\varepsilon\in\mathbb{R}^n$ with $\|\varepsilon\|> 0$ such that $h_i(x+\varepsilon)<0,i\in[m+1,m+\ell]$. Setting $r=-A\varepsilon $ completes the proof.
\end{proof}

To analyze the convergence behavior over convex set $\mathcal{X}$, we also make a regularity assumption for problem~\eqref{P1}, which is commonly used in variational inequalities analysis~\cite{facchinei2007finite}. Two useful definitions are given below and the regularity assumption is given in \cref{ass:crcq}.

\begin{defi}\label{def:Q}
    The constraint matrix $Q(x)$ is defined as $Q(x)=[A^T J^T(x)]^T$.
\end{defi}

\begin{defi}\label{def:Act}
    For $x\in\mathcal{X}$, the index set for the active inequality constraints at $x$ is defined as $\mathcal{I}_x=\{i\in [m+1: m+\ell]\mid h_i(x)=0\}.$ Further, the index set of all active constraints at $x$ is defined as $\mathcal{S}_x=[m]\bigcup \mathcal{I}_x$.
\end{defi}

Recall that for a nonempty index set $\mathcal{S}\subseteq [m+\ell]$, matrix $Q_{\mathcal{S}}(x)$ is defined to be the row submatrix of $Q(x)$ corresponding to the index set $\mathcal{S}$. The regularity assumption is given below. 
\begin{asp}\label{ass:crcq}
    There exists a neighborhood of $x\in\mathcal{X}^*$ such that $Q_{\mathcal{S}_x}(x)$ has the same rank.
\end{asp}
\cref{ass:crcq} is known as the constant rank constraint qualification (CRCQ) condition~\cite{janin1984directional}, which regularizes the solution set $\mathcal{X}^*$. \cref{ass:crcq} implies that, for any $x^*\in \mathcal{X}^*$, there is a neighborhood of $x^*$, such that for any $x$ in this neighborhood, every $Q_{\mathcal{S}_x}(x)$ has a constant rank. Also note that this CRCQ assumption is weaker than the linear independent constraint qualification (LICQ)~\cite{solodov2010constraint}.

\subsection{Convergence Result}
In the following theorem, we estanblish the convergence guarantee of \cref{Alg:SProxALM}.

\begin{thm}\label{thm:main}
Consider solving problem~\eqref{P1} by \cref{Alg:SProxALM} and suppose \cref{ass:basic} and \cref{ass:crcq} hold. Let us choose $p,\ \gamma,\ c,\ \alpha$ satisfying
	\[
	p\ge 3L_f,\quad \gamma\geq 0, \quad c<1/(L_f+\gamma\sigma_{\rm{max}}^2(A)+p),\quad \alpha<\frac{c(p-L_f)^2}{4\sigma_{\rm{max}}^2(A)}.
	\]
	Then, there exists $\beta^\prime>0$ such that for $\beta<\beta^\prime$, the following results hold:
\begin{enumerate}
\item Every limit point of $\{x^t, y^t\}$ generated by \cref{Alg:SProxALM} is a KKT point of problem \eqref{P1};
\item An $\epsilon$-stationary solution can be attained by \cref{Alg:SProxALM} within $\mathcal{O}(1/\epsilon^{2})$ iterations.
\end{enumerate}
\end{thm}

The proof of \cref{thm:main} is presented in \cref{sec:analysis}. The constant $\beta'$ in \cref{thm:main} denpends on the constants in the two established dual error bounds (\cref{weak-eb-ite} and~\cref{local-eb}). These two dual error bounds hold locally and hence $\beta'$ depends on the local structure of problem~\eqref{P1}. Also, the achieved $\mathcal{O}(1/\epsilon^{2})$ iteration complexity is the best known iteration complexity for problem~\eqref{P1}.

\section{Convergence Analysis}\label{sec:analysis}
\subsection{The Potential Function}
Convergence analysis involves two auxiliary problems associated with problem~\eqref{P1}, defined as 
\begin{subequations}\label{eq:def_xyz}
\begin{align}
d(y, z)=\min_{x\in \mathcal{X}}\ K(x, z;y),\quad
&x(y, z)=\arg\min_{x\in \mathcal{X}}\ K(x, z;y),\label{eq:subopt_xyz}\\
P(z)=\min_{x\in \mathcal{X}, Ax=b}\ f(x)+\frac{p}{2}\|x-z\|^2,\quad &{x}(z)=\arg\min_{x\in \mathcal{X}, Ax=b}\ f(x)+\frac{p}{2}\|x-z\|^2\label{eq:subopt_xz}.
\end{align}
\end{subequations}

Denote $\mu_i$ as the Lagrangian multiplier for the inequality constraint $h_i(x)\leq 0$, $i\in[m+1,m+\ell]$, then the KKT conditions for $x(y, z)$ and ${x}(z)$ are given below:
\begin{equation}
\begin{aligned}\label{eq:KKT_xyz}
\nabla_x f(x(y, z))+A^Ty+\gamma A^T(Ax(y, z)-b)+&p(x(y, z)-z)+J^T(x(y, z))\mu=0,\\
\mu_i&\geq 0,\  \forall i\in [m+1: m+\ell],\\
h_i(x(y, z))&\leq 0 ,\  \forall i\in [m+1: m+\ell],\\
\mu_i h_i(x(y, z))&=0,\  \forall i\in [m+1: m+\ell],
\end{aligned}
\end{equation}
and
\begin{equation}
\begin{aligned}\label{eq:KKT_xz}
\nabla_x f({x}(z))+A^Ty+\gamma A^T(Ax( z)-b)+&p({x}(z)-z)+J^T({x}(z))\mu=0,\\
A{x}(z)-b&=0,\\
\mu_i&\geq0,\  \forall i\in [m+1: m+\ell],\\
h_i({x}(z))&\leq 0,\  \forall i\in [m+1: m+\ell],\\
\mu_i h_i({x}(z))&=0\  \forall i\in [m+1: m+\ell].
\end{aligned}
\end{equation}
Since problem \eqref{eq:subopt_xz} depends on $z$, we use $\mathcal{Y}^*(z)$ to denote the set for $y$ satisfying KKT conditions \eqref{eq:KKT_xz}. The following lemma shows the intrinsic relations between $x(y,z)$, ${x}(z)$, and $\mathcal{X}^*$.
\begin{lem}\label{lem:equi_xyz}
	Let $x(y,z)$  and ${x}(z)$ be defined in \eqref{eq:def_xyz}. Then,
	\begin{enumerate}
		\item if $z\in \mathcal{X}^*$, then ${x}(z)=z\in \mathcal{X}^*$;\label{ass:xz=z}
		\item if $z\in \mathcal{X}^*$ and $y\in\mathcal{Y}^*(z)$, then $Ax(y, z)-b=0$;\label{ass:Axyz-b=0}
		\item if $Ax(y, z)-b=0$, then $x(y, z)={x}(z)$;\label{ass:Ax-b=0}
		\item if ${x}(z)=z$, then $z\in \mathcal{X}^*$.\label{ass:zinX}
	\end{enumerate}
\end{lem}
Simply checking the KKT conditions in \eqref{eq:KKT_xyz} and \eqref{eq:KKT_xz} can establish \cref{lem:equi_xyz}. Since the proof is straightforward, we omit it here.

Note that $P(z)$ is known as Moreau envelope in the literature. One important property of $P(z)$ is that the original problem \eqref{P1} can be equivalently reformulated as an unconstrained minimization problem for $P(z)$:
$$\min_{z}P(z)\quad \textrm{s.t.}\ z\in \mathbb{R}^n.$$
If $z^*$ minimizes the above unconstrained problem, then $x(z^*)$ defined in~\eqref{eq:def_xyz} is one stationary point of problem \eqref{P1}. The above unconstrained problem can be solved by performing gradient descent. By Danskin's Theorem~\cite{cvxana}, the gradient of $P(z)$ is given by
$$\nabla P(z)=p(z-{x}(z)).$$ 
Evaluating $\nabla P(z)$ requires solving problem~\eqref{eq:subopt_xz} for $x(z)$, which may be computationally costly. Approximation for $x(z)$ with cheap computational cost is considered. In this way, the update for $(x^t, y^t)$ in \cref{Alg:SProxALM} (lines 4 and 5) can be viewed as one primal-dual step to solve~\eqref{eq:subopt_xz}. The obtained $x^{t+1}$ can be viewed as an estimate of ${x}(z^t)$ with a \textit{primal error} $x^{t+1}-x(y^{t+1}, z^t)$ and a \textit{dual error} $x(y^{t+1}, z^t)-{x}(z^t)$. Note that each primal-dual step tries to reduce the value of the primal-dual potential function $K(x^t, z^t;y^t)-d(y^t, z^t)$.
Also, the update of $z^t$ can be regarded as an approximate gradient descent step of minimizing $P(z)$, and hence the Moreau envelope $P(z)$ can be regarded as a potential function for the $z$-update.

With such intuition, we make use of the following potential function:
$$\phi^t=K(x^t, z^t; y^t)-2d(y^t, z^t)+2P(z^t),$$
where the constant $2$ is for technical convenience and can be replaced by a constant greater than $1$. We hope that this potential function is always decreasing and bounded below. In fact, we have $\phi^t\ge \underline{f}$~\cite{zhang2018proximal}.
Therefore, the key is to prove that $\phi^t$ is decreasing.

\subsection{Three Descent Lemmas}\label{sub:3descent}
In this subsection, we give three basic descent lemmas that are needed to establish the convergence of \cref{Alg:SProxALM}. Particularly, the three descent lemmas estimate the changes of the primal function $K(x,z;y)$, the dual function $d(y,z)$, and the proximal function $P(z)$ after one iteration of \cref{Alg:SProxALM}. We remark that these three lemmas were proved for $\mathcal{X}$ being box constraint~\cite{zhang2018proximal} and polyhedron~\cite{zhang2020global}. Here we will show that they also hold for convex set $\mathcal{X}$. Since the proof is similar to that in~\cite{zhang2018proximal,zhang2020global}, details are presented in \cref{app:3descent}. Let $(x^t,y^t,z^t)$ be generated by \cref{Alg:SProxALM}, the three descent lemmas are: 
\begin{lem}[Primal Descent] \label{primal}
For any $t\geq0$, if $c<1/(L_f+\gamma\sigma_{\rm max}^2(A)+p)$, then 
\begin{equation}\label{eq:primal_descent}
    K(x^t, z^t; y^t)-K(x^{t+1}, z^{t+1}; y^{t+1})  \ge \frac{1}{2c}\|x^t-x^{t+1}\|^2+\frac{p}{2\beta}\|z^t-z^{t+1}\|^2-\alpha\|Ax^t-b\|^2.
\end{equation}
\end{lem}

\begin{lem}[Dual Ascent]\label{dual ascent}
For any $t\geq0$, we have
\begin{equation}
\begin{aligned}
d(y^{t+1}, z^{t+1})-d(y^t, z^t)
\ge & \alpha(Ax^t-b)^T(Ax(y^{t+1}, z^t)-b)\\
&+\frac{p}{2}(z^{t+1}-z^t)^T(z^{t+1}+z^t-2x(y^{t+1}, z^{t+1})).
\end{aligned}
\end{equation}
\end{lem}

\begin{lem}[Proximal Descent]\label{proximal-descent}
For any $t\ge 0$, we have
\begin{equation}\label{eq:prox-descent}
P(z^{t+1})-P(z^{t})\le p(z^{t+1}-z^t)^T(z^t-{x}(z^t))+\frac{p}{2}(\frac{p}{p-L_f}+1)\|z^t-z^{t+1}\|^2.
\end{equation}
\end{lem}

The above three descent lemmas are still not sufficient to show the decrease of the potential function $\phi$. The missing step is to bound the \textit{primal error} $x^{t+1}-x(y^{t+1}, z^t)$ and the \textit{dual error} $x(y^{t+1}, z^t)-{x}(z^t)$ which needs the primal and dual error bounds studied next.

\subsection{The Primal and Dual Error Bounds}
\subsubsection{The Primal Error Bounds}
The primal error bounds are given in the following lemma.
\begin{lem}[Primal Error Bounds]\label{error bound}
	Suppose $p>L_f$, $\gamma>0$ are fixed. Then there exists positive constants $\sigma_1,\sigma_2,\sigma_3,\sigma_4 >0$ (independent of $y$ and $z$) such that the following error bounds hold:
	\begin{eqnarray}
	\|x^{t+1}-x^t\|&\ge& \sigma_1\|x^t-x(y^{t+1}, z^t)\|,\label{eb1}\\
	\|x^{t+1}-x^t\|&\ge& \sigma_2\|x^{t+1}-x(y^{t+1}, z^t)\|, \label{eb2} \\
	\|y-y'\|&\ge&\sigma_3\|x( y,z)-x(y',z)\|, \ \forall y, y', \label{eb6}\\
	\|z^t-z^{t+1}\| &\ge &{\sigma}_4 \|{x}(z^t)-{x}(z^{t+1})\|,\label{eb4} \\
	\|z^t-z^{t+1}\|&\ge &  \sigma_4\|x(y^{t+1}, z^t)-x(y^{t+1}, z^{t+1})\|,\label{eb3}
	\end{eqnarray}
	where $\sigma_1=c(p-L_f)$, $\sigma_2={\sigma_1}/({1+\sigma_1})$, $\sigma_3=(p-L_f)/\sigma_{\rm{max}}(A)$, and  ${\sigma}_4=(p-L_f)/p$.
\end{lem}
The proof can be seen in \cite{zhang2018proximal}.
Using \cref{primal}-\cref{error bound}, we have the following basic estimate for the difference between $\phi^t$ and $\phi^{t+1}$:
\begin{lem}\label{four terms}
	Let us choose $p,\ c,\ \alpha,\ \beta$ satisfying
	\[
	p\ge 3L_f,\quad \gamma\geq 0, \quad c<1/(L_f+\gamma\sigma_{\rm{max}}^2(A)+p),\quad \alpha<\frac{c(p-L_f)^2}{4\sigma_{\rm{max}}^2(A)},\quad \beta<1/24.
	\]
	Then for any $t>0$, we have
	\begin{equation*}
	\begin{aligned}
	\phi^t-\phi^{t+1}
	\ge&\frac{1}{4c}\|x^t-x^{t+1}\|^2+\alpha\|Ax(y^{t+1}, z^t)-b\|^2+\frac{p}{3\beta}\|z^t-z^{t+1}\|^2\\
	\quad &-12p\beta\|x(y^{t+1}, z^t)-{x}(z^t)\|^2.
	\end{aligned}
	\end{equation*}
\end{lem}

The proof is similar to that in \cite{zhang2018proximal} and is presented in \cref{app:fourterms}. According to this lemma, we know that to ensure a sufficient decrease of $\phi^t$, we only need to bound the negative term $-\|x(y^{t+1}, z^t)-{x}(z^t)\|$. To do this, two novel dual error bounds are established.

\subsubsection{The Dual Error Bounds}
The first dual error bound is used to bound the dual error $\|x(y^{t+1}, z^t)-{x}(z^t)\|$ when the residual $\|Ax(y^{t+1}, z^t)-b\|$ is sufficiently small. It is in a non-homogeneous form and we call it \textit{weak} dual error bound.          
\begin{lem}[Weak Dual Error Bound]\label{weak-eb-ite}
Let $\Delta$ be the constant in \cref{prop:slater}, then there exists a constant  $\sigma_w>0$ such that if  $\|Ax(y^{t+1}, z^t)-b\|\le \frac{\Delta}{2}$, we have 
$$\|x(y^{t+1}, z^t)-{x}(z^t)\|^2\le \sigma_w\|Ax(y^{t+1}, z^t)-b\|.$$
\end{lem}

Similar to \cite{svm,gebkkt}, this weak dual error bound is not homogeneous, i.e., the left-hand-side of the inequality has a quadratic error term $\|x(y,z)-{x}(z)\|$ but the right-hand-side has a first-order error term $\|Ax(y,z)-b\|$.
Only using this weak dual error bound is not enough to establish $\mathcal{O}(1/\epsilon^{2})$ iteration complexity of \cref{Alg:SProxALM}. By conducting a nontrivial perturbation analysis, we will further show that when the the optimization residuals of $x, y, z$ are all sufficiently small, the following \textit{strong} dual error bound holds. 

\begin{lem}[Strong Dual Error Bound]\label{local-eb}
	There exists constants $\delta>0$ and $\sigma_s>0$, such that if $\max\left\{\|x^t-x^{t+1}\|, \|Ax(y^{t+1}, z^t)-b\|, \|z^t-x^{t+1}\|\right\}\leq\delta$, 
	we have 
	$$\|x(y^{t+1}, z^t)-x(z^t)\|\leq\sigma_s\|Ax(y^{t+1}, z^t)-b\|.$$
\end{lem}

Note that this lemma is different from Theorem 4.1 in \cite{zhang2020global} which requires $\mathcal{X}$ to be a polyhedral. The proof of Theorem 4.1 in \cite{zhang2020global} relies on the Hoffman bound for polyhedral set. However, here we need to deal with nonlinear constraints and Hoffman bound is not applicable. The nonlinearity is the main challenge in our proof. The proof of the two dual error bounds are presented in \cref{sec:weakeb} and~\cref{sec:strongeb}, respectively. Next, we continue the proof of \cref{thm:main}.

\subsection{Sufficient Decrease of Potential Function}
Combining the weak and strong error bounds, we can dereive the $\mathcal{O}(1/\epsilon^{2})$ iteration complexity. First, we have the following sufficient decrease of potential function $\phi^t$.
\begin{lem}\label{suff-decrease}
	Let us choose $p,\ \gamma,\ c,\ \alpha$ satisfying
	\[
	p\ge 3L_f,\quad \gamma\geq 0, \quad c<1/(L_f+\gamma\sigma_{\rm{max}}^2(A)+p),\quad \alpha<\frac{c(p-L_f)^2}{4\sigma_{\rm{max}}^2(A)}.
	\]
	Then, there exists $\beta^\prime>0$ such that for all $\beta\leq \beta^\prime$, we have
	$$\phi^t-\phi^{t+1}\ge \frac{1}{8c}\|x^t-x^{t+1}\|^2+\frac{\alpha}{2}\|Ax(y^{t+1}, z^t)-b\|^2+\frac{p}{6\beta}\|z^t-z^{t+1}\|^2.$$
\end{lem}


The proof is presented in \cref{app:suff-decrease}. Now we can prove the main theorem.

\medskip
\noindent \textbf{Proof of \cref{thm:main}:} We first prove the convergence of \cref{Alg:SProxALM}. For $x, y, z$, we define $F$ as a map such that $F(x, y, z)=(x^+, y^+, z^+)$, where $(x^+, y^+, z^+)$ is the next iteration point of \cref{Alg:SProxALM}.
It is straightforward to check that the map $F$ is continuous. Also, if $(x,y,z)$ is a fixed point of $F$, $F(x,y,z)=(x,y,z),$
then $(x,y)$ is a pair of primal-dual stationary solution of problem~\eqref{P1}. Suppose that
$$(x^t, y^t, z^t)\rightarrow (\bar{x}, \bar{y}, \bar{z}) \mbox{ along a subsequence $t\in {\cal T}$}.$$
Notice that by \cref{suff-decrease} and $\phi^t\ge \underline{f}$, we have
\[
\|x^t-x^{t+1}\|\rightarrow 0,\quad \|Ax(y^{t+1}, z^t)-b\|\rightarrow 0,\quad
\|z^t-z^{t+1}\|\rightarrow 0.
\]
This further implies
\begin{equation}\label{eq:xyztozero}
\|(x^{t+1}, y^{t+1}, z^{t+1})-(x^t,y^t,z^t)\|\to 0.
\end{equation}
Therefore, we obtain
\begin{eqnarray*}
\|F(\bar{x}, \bar{y}, \bar{z})-(\bar{x}, \bar{y}, \bar{z})\|&=&\lim_{t\rightarrow \infty,\ t\in {\cal T}}\|(x^t, y^t, z^t)-F(x^t, y^t, z^t)\|\\
&=&\lim_{t\rightarrow \infty,\ t\in {\cal T}}\|(x^{t+1}, y^{t+1}, z^{t+1})-F(x^t, y^t, z^t)\|\\
&=&0,
\end{eqnarray*}
where the first step is due to the continuity of $F$ and the second step follows from \eqref{eq:xyztozero}.
Hence, every limit point $(\bar{x}, \bar{y})$ is a pair primal-dual stationary solution of problem~\eqref{P1}.


Next we prove $\mathcal{O}(1/\epsilon^2)$ iteration complexity. 
It follows that for $t>0$, we have $\phi^t\ge \underline{f}$. Then, $\sum_{s=0}^{t-1}(\phi^s-\phi^{s+1})=\phi^0-\phi^t
\le\phi^0-\underline{f}$.
Hence, there exists an $s\in\{0, \cdots, t-1\}$ such that
\begin{equation}\label{eq:decrease rate}
\phi^s-\phi^{s+1}\le (\phi^0-\underline{f})/t.
\end{equation}
Let $C=(\phi^0-\underline{f})\cdot \max\{8c, 2/\alpha, 6\beta/p\}$, then it follows from \cref{suff-decrease} and \eqref{eq:decrease rate} that
\begin{equation}\label{eq:4.3}
\|x^s-x^{s+1}\|^2<C/t,\quad
\|Ax(y^{s+1}, z^s)-b\|^2<C/t,\quad
\|x^{s+1}-z^s\|^2<C/t.
\end{equation}
According to \cref{Alg:SProxALM}, we have
$$x^{s+1}= \arg\min_{x}\left\{\langle \nabla_x K(x^s, z^s; y^{s+1}), x-x^s \rangle+\frac{1}{c}\|x-x^s\|^2+\iota(x)\right\}.$$
The corresponding optimality condition is given by $0\in \nabla_x K(x^s, z^s;  y^{s+1})+\frac{2}{c}(x^{s+1}-x^s)+\partial{\iota(x^{s+1})}.$
Defining $v=\nabla_xK(x^{s+1}, z^s; y^{s+1})-\nabla_xK(x^s, z^s; y^{s+1})-\frac{2}{c}(x^{s+1}-x^s)-\gamma A^T(Ax^{s+1}-b)-p(x^{s+1}-z^s),$ we can rewrite the optimality condition as
$$
v\in \nabla_xK(x^{s+1}, z^s; y^{s+1})-\gamma A^T(Ax^{s+1}-b)-p(x^{s+1}-z^s)+\partial{\iota(x^{s+1})}.
$$
Recalling the definition of $K(x,z;y)$ in \eqref{eq:K}, we have
\begin{equation}\label{eq:gradK}
\nabla_x K(x^{s+1},z^s;y^{s+1})=\nabla f(x^{s+1})+A^Ty^{s+1} +\gamma A^T(Ax^{s+1}-b)+p(x^{s+1}-z^s).
\end{equation}
Therefore, the optimality condition can be further simplified as
\begin{equation}
v\in \nabla f(x^{s+1})+A^Ty^{s+1}+\partial{(\iota(x^{s+1}))}.
\end{equation}
We now proceed to estimate the size of $v$.
By using the triangle inequality and then using the inequalities~\eqref{eb2} and~\eqref{eq:4.3}, we have
\begin{eqnarray}
\|Ax^{s+1}-b\|&\le&\|Ax(y^{s+1}, z^s)-b\|+\|A(x^{s+1}-x(y^{s+1}, z^s))\|\nonumber\\
&\le&\frac{\sqrt{C}}{\sqrt{t}}+\sigma_{\rm{max}}(A)\frac{1}{\sigma_2} \frac{\sqrt{C}}{\sqrt{t}}\nonumber\\
&=&\frac{\sqrt{B_1C}}{\sqrt{t}},\label{eq:4.4}
\end{eqnarray}
where $B_1=\left(1+\sigma_{\rm{max}}(A)\frac{1}{\sigma_2}\right)^2>0$. For $(x^{s},z^s;y^{s+1})$, we have
$$
\nabla_x K(x^{s},z^s;y^{s+1})=\nabla f(x^{s})+A^Ty^{s+1} +\gamma A^T(Ax^{s}-b)+p(x^{s}-z^s).
$$
Combining it with~\eqref{eq:gradK} leads 
\begin{equation*}
    \begin{aligned}
\nabla_x K(x^{s+1},z^s;y^{s+1})-\nabla_x K(x^{s},z^s;y^{s+1})=&\nabla f(x^{s+1})-\nabla f(x^{s})\\
&+(\gamma A^TA+p I)(x^{s+1}-x^s).
    \end{aligned}
\end{equation*}
This further implies
\begin{equation*}
    \begin{aligned}
    &\|\nabla_x K(x^{s+1},z^s;y^{s+1})-\nabla_x K(x^{s},z^s;y^{s+1})-\frac{2}{c}(x^{s+1}-x^s)\|\\
    &\le (L_f+p+\gamma\sigma_{\rm{max}}^2(A)+2/c)\|x^{s+1}-x^s\|,
    \end{aligned}
\end{equation*}
where the Lipschitz continuity of $\nabla f(x)$ is used. Then we have
\begin{eqnarray}
\|v\|&\le& (L_f+p+\gamma\sigma_{\rm{max}}^2(A)+2/c)\|x^s-x^{s+1}\|+\gamma\|Ax^{s+1}-b\|+p \|x^{s+1}-z^s\|\nonumber\\
&\le& (L_f+p +\gamma\sigma_{\rm{max}}^2(A)+2/c)\frac{\sqrt{C}}{\sqrt{t}}+\gamma\sigma_{\rm{max}}(A)\frac{\sqrt{B_1C}}{\sqrt{t}}+p\frac{\sqrt{C}}{\sqrt{t}}\nonumber\\
&\le&\sqrt{B_2C}/\sqrt{t},\label{eq:optv}
\end{eqnarray}
where the second inequality follows from inequalities \eqref{eq:4.3} and \eqref{eq:4.4}, and
$$B_2=((L_f+p+\gamma\sigma_{\rm{max}}(A)^2+2/c)+\gamma\sigma_{\rm{max}}(A)\sqrt{B_1}+p)^2>0.$$
Denote $B=C\max\{B_1, B_2\}$, then~\eqref{eq:4.4} and~\eqref{eq:optv} implies $(x^{s+1}, y^{s+1})$ is a $\sqrt{B}/\sqrt{t}$-stationary solution which completes the proof.


\section{Proof of Weak Dual Error Bound (\cref{weak-eb-ite})}\label{sec:weakeb}
In this section, we first show that if $\|Ax(y, z)-b\|$ is small then the dual variable $y\in y^0+\mathrm{Range}(A)$ is always bounded. This together with strong convexity of $K(\cdot,y,z)$ and Lipschiz continuity of $\nabla K(\cdot,y,z)$ leads the weak dual error bound in \cref{weak-eb-ite}.
\begin{lem}\label{lem:y_bounded}
Let $\Delta$ be the constant in \cref{prop:slater}, then for any $z\in \mathcal{X}$ and $y\in y^0+\mathrm{Range}(A)$, there exists a constant $U>0$ such that if $\|Ax(y, z)-b\|\leq \frac{\Delta}{2},$
we have $\|y\|<U$, where $U=\|y^0\|+4U_0/\Delta$ with $U_0=\max_{x,z\in \mathcal{X}}\ |K(x, z; y^0)|$.
\end{lem}
\begin{proof}
Let $\tilde{y}\in\mathrm{Range}(A)$, we have 
\begin{eqnarray}\label{1/2}
\tilde{y}^T(Ax(y, z)-b)\stackrel{ \mbox{\scriptsize(i)}}\geq-\|\tilde{y}\|\cdot \|Ax(y, z)-b\|\stackrel{ \mbox{\scriptsize(ii)}}\geq-\|\tilde{y}\|\Delta/2,
\end{eqnarray}
where (i) is due to the Cauchy-Schwarz inequality and (ii) is because $\|Ax(y, z)-b\|<\Delta/2$. By the definition of $\Delta$ in \cref{prop:slater}, we have that for any $r\in \mathrm{Range}(A)$ with $\|r\|\le \Delta$, there exists some $x\in \mathcal{X}$ satisfying $Ax-b=r$. Since $\tilde{y}\in\mathrm{Range}(A)$, then we have some $\tilde{x}\in \mathcal{X}$ satisfying $A\tilde{x}-b=-\Delta\tilde{y} / \|\tilde{y}\|$. Denote $U_0=\max_{x,z\in \mathcal{X}}\ |K(x, z; y^0)|$, by the definition of $x(y, z)$, we have
\begin{equation*}
\begin{aligned}
0&\leq K(\tilde{x}, z;y)-K(x(y, z), z;y)\\
&\stackrel{ \mbox{\scriptsize(i)}}=(K(\tilde{x}, z;y^0)+\tilde{y}^T(A\tilde{x}-b))-(K(x(y, z), z;y^0)+\tilde{y}^T(Ax(y, z)-b))\\
&\stackrel{ \mbox{\scriptsize(ii)}}\leq2U_0+\tilde{y}^T(A\tilde{x}-b)-\tilde{y}^T(Ax(y, z)-b)\\
&\stackrel{ \mbox{\scriptsize(iii)}}\leq2U_0-\|\tilde{y}\|\Delta/2,
\end{aligned}
\end{equation*}
where (i) is due to the definition of $K(\cdot)$, (ii) is by the definition of $U_0$, and (iii) is because $A\tilde{x}-b=-\Delta\tilde{y} / \|\tilde{y}\|$ and \eqref{1/2}. Hence, we have $\|\tilde{y}\|\le 4U_0/\Delta$. Finally, by the triangular inequality, we further have $\|y\|\le \|y^0\|+4U_0/\Delta.$ Setting $U=\|y^0\|+4U_0/\Delta$ completes the proof.
\end{proof}

Next, recall the set $\mathcal{Y}^*(z)\subseteq\mathbb{R}^m$ which represents the solution set of dual variable $y$ of the KKT conditions~\eqref{eq:KKT_xz}. Then we have the following lemma. 
\begin{lem}
For any $z\in\mathcal{X}$, there exists at least one $y(z)\in\mathcal{Y}^*(z)$ such that $y(z)\in y^0+\mathrm{Range}(A)$.
\end{lem}
\begin{proof}
For any $y(z)\in\mathcal{Y}^*(z)$ with $z\in \mathcal{X}$, we can decompose it as $y(z)=y_N(z)+y_R(z)$, where $y_N(z)\in \mathrm{Null}(A)$ and $y_R(z)\in\mathrm{Range}(A)$. 
Similarly, the vector $y^0$ can be decomposed as $y^0=y_N^0+y_R^0$ with $y_N^0\in \mathrm{Null}(A)$ and $y_R^0\in\mathrm{Range}(A)$. Define $y^\prime(z)=y^0+y_R(z)$, then by the definition of $\mathcal{Y}^*(z)$ (c.f. \eqref{eq:KKT_xz}), we have $y^\prime(z)\in\mathcal{Y}^*(z)$. 
Since $y^{\prime}(z)-y^0=y_R(z)\in \mathrm{Range}(A)$, we have $y^\prime(z)\in y^0+\mathrm{Range}(A)$.
\end{proof}

Therefore, combining this lemma with \cref{lem:y_bounded}, we have $\|y(z)\|\leq U$ (by \cref{lem:y_bounded}) which can be used to prove \cref{weak-eb-ite}.

\medskip
\noindent \textbf{Proof of \cref{weak-eb-ite}:} 
Let $y\in y^0+\mathrm{Range}(A)$ and $y(z)\in\mathcal{Y}^*(z)$ satisfying $y(z)\in y^0+\mathrm{Range}(A)$ and $z\in \mathcal{X}$. Then, by the strong convexity of $K(x, z;y)$ in $x$, we have
\begin{equation}
\begin{aligned}\label{eq:K_sc}
(p-L_f)\|x(y, z)-{x}(z)\|^2&\leq K(x(y, z), z;y(z))-K({x}(z), z;y(z)),\\
(p-L_f)\|x(y, z)-{x}(z)\|^2&\leq K({x}(z), z;y)-K(x(y, z), z;y).
\end{aligned}
\end{equation}
Also, by the definition of $K(x, z;y)$, we have
\begin{equation}
\begin{aligned}\label{eq:K_def}
K(x(y, z), z;y)-K(x(y, z), z;y(z))&=\langle Ax(y, z)-b, y-y(z)\rangle,\\
K({x}(z), z;y)&=K({x}(z), z;y(z)).
\end{aligned}
\end{equation}
Combining \eqref{eq:K_sc} and \eqref{eq:K_def}, we obtain 
\begin{equation}\label{eq:weak1}
\|x(y, z)-{x}(z)\|^2\le  \frac{1}{2(p-L_f)}\langle y-y(z), Ax(y, z)-b\rangle,
\end{equation}
By \cref{lem:y_bounded}, we have $\|y\|\le U$ and $\|y(z)\|\le U$, which implies $\|y-y(z)\|\le 2U$. Applying Cauchy-Schwarz inequality to \eqref{eq:weak1}, we have $\|x(y, z)-{x}(z)\|^2\le 2U\zeta \|Ax(y, z)-b\|$. Setting $\sigma_w=2U\zeta$ completes the proof.

\section{Proof of Strong Dual Error Bound (\cref{local-eb})}\label{sec:strongeb}
In this section, we first give some technical lemmas, which state basic properties of $x(y, z)$ and ${x}(z)$ near the solution set $\mathcal{X}^*$. Then, based on these properties we introduce a perturbation bound which can be used to prove \cref{local-eb}.

\subsection{Technical Lemmas}
\begin{lem}\label{continuity_xyz}
$x(y, z)$ is continuous in $(y, z)$ and ${x}(z)$ is continuous in $z$.
\end{lem}
\begin{proof}
This lemma is a direct corollary of \eqref{eb6}, \eqref{eb4}, and \eqref{eb3} in \cref{error 
bound}.
\end{proof}
\begin{lem}\label{compactness}
The solution set $\mathcal{X}^*$ is closed and hence is compact.
\end{lem}
\begin{proof}
Suppose $\{{x}^i\}\subseteq \mathcal{X}^*$ is a sequence converging to $\bar{x}\in \mathcal{X}$. We prove that $\bar
{x}\in \mathcal{X}^*$ and $A\bar{x}=b$. Let $\bar{z}^i={x}^i\in\mathcal{X}^*$, then by \cref{lem:equi_xyz}.\ref{ass:xz=z} we have 
${x}(\bar{z}^i)=\bar{z}^i\in\mathcal{X}^*$. Without loss of generality, suppose $\bar{z}^i\rightarrow \bar{z}$, \cref{continuity_xyz} implies ${x}(\bar{z})=\bar{z}$. Therefore, according to \cref{lem:equi_xyz}.\ref{ass:zinX}, we have $\bar{x}\in \mathcal{X}^*$. Let $\bar{y}^i\in Y(\bar{z}^i)\cap(y^0+\mathrm{Range}(A))$, then by \cref{lem:equi_xyz}.\ref{ass:Axyz-b=0}, we have 
$Ax(\bar{y}^i,\bar{z}^i)-b=0$. This together with \cref{lem:equi_xyz}.\ref{ass:Ax-b=0} implies $x(\bar{y}^i,\bar{z}^i)={x}
(\bar{z}^i)$. By \cref{lem:y_bounded}, $\{\bar{y}^i\}$ is bounded and therefore both $\{\bar{y}^i\}$ and $\{\bar{z}^i
\}$ have limit points  $\bar{y}$ and $\bar{z}$ respectively. Further, since $\bar{z}^i={x}^i$, we have $\bar{z}^i={x}
(\bar{z}^i)={x}^i=x(\bar{y}^i, \bar{z}^i)$. This completes the proof.
\end{proof}

Based on \cref{continuity_xyz} and \cref{compactness}, we can show that when the optimization residuals are small, $x(y, z)$ and ${x}(z)$ are close to the solution set $\mathcal{X}^*$.
\begin{lem}\label{lem:epsilon}
For any $\epsilon>0$, there exists a constant $\delta(\epsilon)>0$ (depending on $\epsilon$) such that if $\max\{\|Ax(y, z)-b\|, \|z-{x}(z)\|\}\leq \delta
(\epsilon)$ for $y\in y^0+\mathrm{Range}(A)$ and $z\in\mathcal{X}$, we have 
$$\max\left\{ \mathrm{dist}(x(y, z), \mathcal{X}^*), \mathrm{dist}({x}(z), \mathcal{X}^*)\right\}\leq \epsilon.$$
\end{lem}

\begin{proof}
We only prove the claim for $x(y, z)$ and the proof for $x(z)$ is similar. We prove it by contradiction. Suppose the contrary, then there exists an $\epsilon>0$, a sequence $\delta^i\rightarrow 0$, $\{\bar{y}^i\}\subseteq   y^0+\mathrm{Range}(A)$ 
and $\bar{z}^i \in \mathcal{X}$ 
 such that 
$$\max\{\|Ax(\bar{y}^i, \bar{z}^i)-b\|, \|\bar{z}^i-{x}(\bar{z}^i)\|\}\leq \delta^i,$$
but $\mathrm{dist}(x(\bar{y}^i, \bar{z}^i), \mathcal{X}^*)>\epsilon$. Without loss of generality, we can assume $\delta^i<
\Delta/2$.
By \cref{lem:y_bounded} and compactness of $\mathcal{X}$, both sequences  $\bar{y}^i$ and $\bar{z}^i$ are bounded.  
Hence, there exist limit points of $\{\bar{y}^i\}$ and $\{\bar{z}^i\}$ respectively. Passing to a  sub-sequence if 
necessary, we can assume that $\bar{y}^i\rightarrow \bar{y}, \bar{z}^i\rightarrow \bar{z}$. By the continuity of $x(y,z)$ 
proved in \cref{continuity_xyz}, we have 
$$x(\bar{y}^i, \bar{z}^i)\rightarrow x(\bar{y}, \bar{z}),\quad  {x}(\bar{z}^i)\rightarrow {x}(\bar{z}).$$
Hence, we have $Ax(\bar{y}, \bar{z})-b=0$ and $\bar{z}={x}(\bar{z}).$ 
According to \cref{lem:equi_xyz} and \cref{compactness}, we immediately have 
$x(\bar{y}, \bar{z})={x}(\bar{z})\in \mathcal{X}^*$, which is a contradiction for
$$\mathrm{dist}(x(\bar{y}^i, \bar{z}
^i), \mathcal{X}^*)>\epsilon.$$
The bound $\mathrm{dist}({x}(z), \mathcal{X}^*)\leq \epsilon$ can be proved similarly.
\end{proof}

\subsection{A Perturbation Error Bound}
To prove \cref{local-eb}, we reduce the dual error bound to an equivalent perturbation bound. We perturb $Ax=b$ in auxiliary problem~\eqref{eq:subopt_xz} as $Ax-b=r$ with $\| r\| \leq \Delta$ and  $r\in \mathrm{Range}(A)$. The solution of such perturbed problem is denoted as $\hat{x}(r, z)$  ($p>L_f$):
\begin{equation}\label{eq:aux_p3}
\hat{x}(r, z)=\arg\min_{x\in \mathcal{X}, Ax-b=r}\ f(x)+\frac{\gamma}{2}\|Ax-b\|^2+\frac{p}{2}\|x-z\|^2.
\end{equation}
By \cref{prop:slater}, the Slater condition holds for problem~\eqref{eq:aux_p3} with $\|r\|\le \Delta$ and $r
\in\mathrm{Range}(A)$. Hence, the KKT conditions hold for $\hat{x}(r, z)$:
\begin{equation}
\begin{aligned}\label{eq:KKT_xrz}
\nabla_xf(\hat{x}(r, z))+A^Ty+\gamma A^T(A\hat{x}(r, z)-b)&+p(\hat{x}(r, z)-z)+J^T(\hat{x}(r, z))=0,\\
A\hat{x}(r, z)-b&=r,\\
\mu_i&\ge 0,\ i\in[m+1: m+\ell],\\
h_i(\hat{x}(r, z))&\le 0,\ i\in[m+1: m+\ell],\\
\mu_ih_i(\hat{x}(r, z))&=0,\ i\in [m+1: m+\ell].
\end{aligned}
\end{equation}

It is easy to see $\hat{x}(r, z)$ enjoys the following relation with $x(y,z)$ and ${x}(z)$.
\begin{lem}\label{lem:xyz=xrz}
For any $y$, $z$, and $r=Ax(y,z)-b$ with $\|r\|\le \Delta$, we have $\hat{x}(r, z)=x(y, z)$ and $\hat{x}(0, z)={x}(z).$
\end{lem}
\begin{proof}
Equality $r=Ax(y,z)-b$ together with KKT conditions for $x(y,z)$ in~\eqref{eq:KKT_xyz} have the same form as KKT conditions 
for $\hat{x}(r, z)$ in~\eqref{eq:KKT_xrz}, except $x(y,z)$ is replaced by $\hat{x}(r, z)$. This implies $\hat{x}(r, z)=x
(y,z)$. Similarly, KKT conditions for $\hat{x}(0, z)$ in~\eqref{eq:KKT_xrz} have the same form as KKT conditions for $\bar
{x}(z)$ in~\eqref{eq:KKT_xz}, except $\hat{x}(0, z)$ is replaced by ${x}(z)$. This implies $\hat{x}(0, z)={x}(z)$.
\end{proof}

\cref{lem:xyz=xrz} implies the difference between $x(y, z)$ and ${x}(z)$ can be reduced to $\hat{x}(r, z)-
\hat{x}(0, z)$. It motivates us to bound the perturbation error term $\|\hat{x}(r, z)-\hat{x}(0, z)\|$ instead of the dual 
error $\|x(y, z)-{x}(z)\|$. In fact, based on \cref{lem:epsilon}, we can bound this perturbation error by the residual $r$ as given in the following 
lemma. 

\begin{lem}\label{local-eb2}
Let $\epsilon_0>0$ and $\Omega$ be a collection of $(r, z)$ defined as 
$$\Omega=\left\{ (r, z)\mid r\in \mathrm{Range}(A),\  z\in\mathcal{X},\  \max\{\|{r}\|,\  \|z-{x}({z})\|\}\le \delta(\epsilon_0),\ \|r\|\le \Delta
\right\}.$$
Then there exists a constant $\sigma_s>0$, such that for any $y\in y^0+\mathrm{range}(A), z\in \mathcal{X}$, and $r=Ax(y, z)-b$ with 
$(r, z)\in \Omega,$
we have
\begin{equation}\label{eb-perbr}
\|\hat{x}(r, z)-\hat{x}(0, z)\|\le \sigma_s\|r\|.
\end{equation}
\end{lem}

The detailed proof of \cref{local-eb2} is presented in \cref{subsec:perb_bound}. Note that this dual error bound~\eqref{eb-perbr} is similar to Assumption 5 in \cite{zeng2021moreau}, i.e., $\|\hat{x}(r, z)-\hat{x}(r', z)\|\le \bar{M}\|r-r'\|,$
where $\bar{M}>0$ is assumed to be smaller than the smallest nonzero eigenvalue of $A^TA$. This assumption is stronger than the claim in \cref{local-eb2}. Beased on CRCQ assumption (\cref{ass:crcq}), we will prove the existence of $\sigma_s$ and our convergence analysis does not need $\sigma_s$ being smaller than a certain threshold. This is the benifit of our potential function. \cref{local-eb2} implies that for $z$ in some neighbourhood of the solution set, the dual error bound~\eqref{eb-perbr} holds uniformly, i.e., it holds for any $z$ satisfying $\|x(z)-z\|\le \delta(\epsilon_0)$. In the literature, this kind of dual error bound is established under CRCQ assumption for a fixed objective function~\cite{solodov2010constraint}. Specializing to our setting, existing result~\cite{solodov2010constraint} implies the dual error bound~\eqref{eb-perbr} holds for any $r$ satisfying $\|r\|\le \delta_z$, where $\delta_z$ is a constant depending on $z$. However, this $\delta_z$ may not have a uniform positive lower bound for $z$ in a neighbourhood of the solution set. Therefore, the existing result~\cite{solodov2010constraint} can not be directly applied to our setting since the objective function is also perturbed. Instead, by introducing the notion of basic set (\cref{def:basicset}), we develop a decomposition technique (\cref{dual-pieces1}) to prove \cref{local-eb2} and details are presented in \cref{subsec:perb_bound}.


\medskip
\noindent \textbf{Proof of \cref{local-eb}:} Let $\epsilon_0>0$ and $\delta>0$ be a constant satisfying
$$(1+\sigma_{\rm max}(A)/\sigma_1)\delta< \min\{\delta(\epsilon_0), \Delta/2\},\ (1+1/\sigma_1)\delta+\sqrt{\sigma_w(1+\sigma_{\rm max}(A)/\sigma_1)\delta}<\delta(\epsilon_0).$$
Suppose $\max\{ \|x^t-x^{t+1}\|,  \|Ax(y^{t+1},z^t)-b\|, \|z^t-x^{t+1}\|\}<\delta$ and denote $\tilde{r}^t=Ax(y^{t+1}, z^t)-b$, we show below that $(\tilde{r}^t, z^t)\in \Omega$.

First, we have
\begin{equation}
\begin{aligned}\label{Omega-condition1}
\|Ax(y^{t+1}, z^t)-b\|&\le\|Ax^{t+1}-b\|+\|A(x^{t+1}-x(y^{t+1}, z^t))\|\\
&\le \delta+\sigma_{\rm max}(A)\|x^{t+1}-x(y^{t+1}, z^t)\|\\
&\stackrel{ \mbox{\scriptsize(i)}}\le\delta+\frac{\sigma_{\rm max}(A)}{\sigma_1}\|x^t-x^{t+1}\|\\
&\le  \delta+\delta\sigma_{\rm max}(A)/\sigma_1\\
&\le\min\{\delta(\epsilon_0), \Delta/2\},
\end{aligned}
\end{equation}
where (i) is due to \eqref{eb1}. Next, we have
\begin{equation}
\begin{aligned}\label{Omega-condition2}
\|z^t-{x}(z^t)\|&\le \|z^t-x^{t+1}\|+\|x^{t+1}-x(y^{t+1}, z^t)\|+\|x(y^{t+1}, z^t)-{x}(z^t)\|\\
&\stackrel{ \mbox{\scriptsize(i)}}\le \delta+\delta/\sigma_1+\sqrt{\sigma_w\|Ax(y^{t+1}, z^t)-b\|}\\
&\stackrel{ \mbox{\scriptsize(ii)}}\le\delta+\delta/\sigma_1+\sqrt{\sigma_w(\delta+\delta\sigma_{\rm max}(A)/\sigma_1)}\\
&\le \delta(\epsilon_0),
\end{aligned}
\end{equation}
where (i) is because  of the primal error bound \eqref{eb1} and the weak dual error bound in \cref{weak-eb-ite} and (ii) is due to \eqref{Omega-condition1}.
Therefore, by \eqref{Omega-condition1} and \eqref{Omega-condition2}, we have $(\tilde{r}^t, z^t)\in \Omega.$ Then by \cref{lem:xyz=xrz}, we have $x(y^{t+1}, z^t)-{x}(z^t)=\hat{x}(\tilde{r}^t, z^t)-\hat{x}(0, z^t).$
Therefore, by \cref{local-eb2}, we have
\begin{eqnarray*}
\|x(y^{t+1}, z^t)-{x}(z^t)\|&=&\|\hat{x}(\tilde{r}^t, z^t)-\hat{x}(0, z^t)\|\\
&\le&\sigma_s\|\tilde{r}^t\|\\
&=&\sigma_s\|Ax(y^{t+1}, z^t)-b\|.
\end{eqnarray*}
This completes the proof.

\section{Convergence of Smoothed Proximal ADMM}
This section extends \cref{Alg:SProxALM} to the multi-block case. Let $f(x_1, x_2, \cdots, x_N)$ be a smooth (possibly nonconvex) function with $L_f$-Lipschitz-continuous gradient. Then, the multi-block optimization problem considered is:
\begin{equation}
\label{P2}
\begin{array}{ll}
\mbox{minimize}& f(x_1, x_2, \cdots, x_N)\\ [5pt]
\mbox{subject to} & \sum_{i=1}^NA_ix_i=b,\ x_i \in \mathcal{X}_i,
\end{array}
\end{equation}
where $\mathcal{X}_i$ is a compact convex set, $A_i\in\mathbb{R}^{m\times n_i}$, $x_i\in \mathcal{X}_i \subseteq \mathbb{R}^{n_i}$, $i\in[N]$, and $x=(x_1, x_2, \cdots, x_N)\in \mathbb{R}^n$.

Convergence of ADMM for problem~\eqref{P2} in general can not be guaranteed. Similar to smoothed Prox-ALM, we use the same smoothing technique and propose a smoothed proximal ADMM which updates the primal variables in an inexact block coordinate descent way. Denote 
$$x^t(i)=(x_1^{t+1}, x_2^{t+1}, \cdots, x_{i-1}^{t+1}, x_i^t, \cdots, x_N^t)$$
and $\mathcal{P}_{\mathcal{X}_i}(\cdot)$ as the projection to the set $\mathcal{X}_i$, then the proposed algorithm is given below.
\begin{algorithm}
\caption{Smoothed Proximal ADMM}
\label{Alg3}
\begin{algorithmic}
\STATE Let $\alpha>0$, $0<\beta\le 1$, and $c>0$;
\STATE Initialize $x^0\in \mathcal{X},\ z^0\in \mathcal{X},\ y^0\in \mathbb{R}^m$;
\FOR{$t=0,1,2,\ldots,$}
\STATE $y^{t+1}=y^t+\alpha(Ax^t-b)$;
\FOR{$i=1, 2, \cdots, N$}
\STATE $x^{t+1}_i=\mathcal{P}_{\mathcal{X}_i}(x^t_i-c\nabla_{x_i} K(x^t(i), z^t; y^{t+1}))$;
\ENDFOR
\STATE $z^{t+1}=z^t+\beta(x^{t+1}-z^t)$. 
\ENDFOR
\end{algorithmic}
\end{algorithm}

Similar to the one-block case, we have the following convergence result for \cref{Alg3}.

\begin{thm}\label{main:multi}
Consider solving problem~\eqref{P2} by \cref{Alg3} and suppose \cref{ass:basic} and \cref{ass:crcq} hold.
Let us choose $p,\ \gamma,\ c,\ \alpha$ satisfying
	\[
	p\ge 3L_f,\quad \gamma\geq 0, \quad c<1/(L_f+\gamma\sigma_{\rm{max}}^2(A)+p),\quad \alpha<\frac{c(p-L_f)^2}{4\sigma_{\rm{max}}^2(A)}.
	\]
	Then, there exists $\beta^\prime>0$ such that for $\beta<\beta^\prime$, the following results hold:
\begin{enumerate}
\item Every limit point of $\{x^t, y^t\}$ generated by \cref{Alg3} is a KKT point of problem \eqref{P2};
\item An $\epsilon$-stationary solution can be obtained by \cref{Alg3} within $\mathcal{O}(1/\epsilon^{2})$ iterations.
\end{enumerate}
\end{thm}

The proof of \cref{main:multi} mostly follows the same line as that in the one-block case. The only differences are the primal descent inequality~\eqref{eq:primal_descent} in \cref{primal} and the two primal error bounds (\eqref{eb1} and \eqref{eb2}) in \cref{error bound}. Therefore, we only need to prove them for the multi-block case. For the primal error bounds \eqref{eb1} and \eqref{eb2}, we only give the proof of the first one and the other one can be proved using the same techniques.
\begin{lem}\label{bareb1-c2}
For any $t\geq 0$, there exists a constant $\bar{\sigma}_1$, such that 
$$\|x^t-x^{t+1}\|\ge \bar{\sigma}_1\|x^t-x(y^{t+1}, z^t)\|.$$
\end{lem}
\begin{proof}
Since this lemma is not related to the update of $y, z$, for notation simplicity, we denote $K(x)=K(x, z^t; y^{t+1})$.
The proof consists of two parts. First we regard the block coordinate gradient descent scheme in the primal step as a type of approximate 
gradient projection algorithm, $x^{t+1}=P_{\mathcal{X}}(x^t-c\nabla K(x^t)+\mathcal{E}(t)),$
where $\mathcal{E}(t)$ satisfies
\begin{equation}\label{Theta}
\|\mathcal{E}(t)\|< \eta\|x^t-x^{t+1}\|
\end{equation}
for some positive constant $\eta>0$. Then, we prove this approximate gradient projection algorithm also has the primal error bound.

For the first part we have 
\begin{eqnarray}
x^{t+1}_j&=&P_{\mathcal{X}_j}(x^t_j-c\nabla_{x_j} K(x^t(j)))\nonumber\\
&=&P_{\mathcal{X}_j}(x^t_j-c\nabla_{x_j} K(x^t)+c(\nabla_{x_j} K(x^t(j))-\nabla_{x_j} K(x^t)))\nonumber\\
&=&P_{\mathcal{X}_j}(x^t-c\nabla_{x_j} K(x^t)+\mathcal{E}_j(t)),\nonumber,
\end{eqnarray}
where $\mathcal{E}_j(t)=c\left(\nabla_{x_j} K(x^t(j))-\nabla_{x_j} K(x^t)\right)$.
Due to the Lipschitz continuity of the partial gradient of $K$, we have
\begin{eqnarray}
\|\mathcal{E}_j(t)\|&\le&c({L}+p+\gamma {\sigma}_{\max}^2(A))\|x^t(j)-x^t\|\nonumber\\
&\le&c({L}+p+\gamma {\sigma}^2_{\max}(A))\sqrt{\sum_{i=1}^{j-1}\|x_i^{t+1}-x_i^t\|^2}\nonumber\\
&\le& c({L}+p+\gamma {\sigma}^2_{\max}(A))\sum_{i=1}^{j-1}\|x_i^t-x_i^{t+1}\|,
\end{eqnarray}
where the last inequality is by squaring both sides of the inequality.
Since 
$$\sum_{i=1}^{j-1}\|x_i^t-x_i^{t+1}\|\le \sum_{i=1}^N\|x_i^t-x_i^{t+1}\|,$$
we have
\begin{eqnarray}
\|\mathcal{E}(t)\|&=&\|(\mathcal{E}_1(t), \mathcal{E}_2(t),\cdots, \mathcal{E}_N(t))\|\nonumber\\
&\le&c({L}+p+\rho{\sigma}_{\max}^2(A))N\sum_{i=1}^N\|x^t_i-x^{t+1}_i\|\nonumber\\
&\le&c({L}+p+\rho{\sigma}^2_{\max}(A))N^{\frac{3}{2}}\|x^t-x^{t+1}\|,\nonumber
\end{eqnarray}
where the last inequality is due to Cauchy-Schwartz inequality.
This finishes the proof of \eqref{Theta} with $\eta=c({L}_f+p+\gamma\sigma^2_{\max}(A))N^{3/2}$.
For the second part, we have
\begin{eqnarray}
\|x^t-x^{t+1}\|&=&\|x^t-P_{\mathcal{X}}(x^{t}-c\nabla K(x^t)+\mathcal{E}(t))\|\nonumber\\
&\stackrel{ \mbox{\scriptsize(i)}}\ge& \|x^t-P_\mathcal{X}(x^t-c\nabla K(x^t))\|-\|P_\mathcal{X}(x^t-c\nabla K(x^t))-P_{\mathcal{X}}(x^t-c\nabla 
K(x^t)+\mathcal{E}(t))\|\nonumber\\
&\stackrel{ \mbox{\scriptsize(ii)}}\ge& \sigma_1\|x^t-x(y^{t+1}, z^t)\|-\|P_{\mathcal{X}}(x^t-c\nabla K(x^t))-P_{\mathcal
{X}}(x^t-c\nabla K(x^t)+\mathcal{E}(t))\|\nonumber\\
&\stackrel{ \mbox{\scriptsize(iii)}}\ge& \sigma_1\|x^t-x(y^{t+1}, z^t)\|-\|\mathcal{E}(t)\|\nonumber\\
&\stackrel{ \mbox{\scriptsize(iv)}}\ge& \sigma_1\|x^t-x(y^{t+1}, z^t)\|-\eta\|x^t-x^{t+1}\|,\nonumber
\end{eqnarray}
where (i) is because of the triangular inequality, (ii) is due to the error bound \eqref{eb1} in \cref{error bound}, 
(iii) is due to the nonexpansiveness of the projection operator and (iv) is because of \eqref{Theta}. Setting $\bar
{\sigma}_1=\sigma_1/(1+\eta)$ completes the proof.
\end{proof}
Next we establish a simple lemma to ensure that the primal descent inequality~\eqref{eq:primal_descent} in \cref{primal} holds for the multi-block case. Since  only the primal update is different, we just need to prove the following primal dscent also holds in the multi-block case.
\begin{lem}\label{primal2}
For any $t\geq 0$, we have 
$$K(x^t, z^t; y^{t+1})-K(x^{t+1}, z^t; y^{t+1})\ge \frac{1}{2c}\|x^t-x^{t+1}\|^2.$$
\end{lem}
\begin{proof}
By \cref{ass:basic}.\ref{ass:lip}, the partial gradient of $K$ with respect to any block is at least $c^{-
1}$-Lipschitz continuous, so we have
$$K(x^t(j), z^t; y^{t+1})-K(x^t(j+1), z^t ;y^{t+1})\ge \frac{1}{2c}\|x^t_j-x^{t+1}_j\|^2,\mbox{ for all }1\le j\le k.$$
Here $x^t(0)=x^t$.
Summing this from $0$ to $k-1$ and using the fact that $\sum_{j=1}^k\|x^t_j-x^{t+1}_j\|^2=\|x^t-x^{t+1}\|^2$ yields the desired result.
\end{proof}

\section{Numerical Results}
Convergence behavior of the proposed smoothed proximal ALM (SProx-ALM for short) for a class of nonconvex quadratic programming (QP) problems, i.e., nonconvex quadratic loss function with linear constraints and a convex $\ell_2$-norm constraint, is presented in this section. In particular, the considered QP problems take the following form:
\begin{equation}\label{opt:numerQP}
\begin{aligned}
    \min_{x}\ \frac{1}{2}x^TQx+r^Tx\quad 
    \textrm{s.t.}\ & Ax=b,\ \| x\|\leq c,
\end{aligned}
\end{equation}
where $Q\in\mathbb{R}^{n\times n}$ is a symmetric matrix (possibly not positive semi-definite), $r\in\mathbb{R}^n$, $A\in\mathbb{R}^{m\times n}$, $b\in\mathbb{R}^n$, and $c>0$ is a positive constant. Different choices of problem size $n$ are considered, i.e., $n=50,100,200$, and $m$ is fixed to be $m=20$. Matrix $Q$ is generated as $Q=(\bar{Q}+\bar{Q}^T)/2$, where each entry of $\bar{Q}$ is sampled from standard Gaussian distribution. Each entry of $r$ and $A$ is also sampled from standard Gaussian distribution and $c$ is sampled from uniform distribution over region $[1,10]$. To ensure feasibility of the generated problem instance, $b$ is generated as $b=A\bar{x}$, where $\bar{x}$ is sampled from standard Gaussian distribution with $\| \bar{x} \|\leq c$. 

The stationary gap of a primal-dual pair $(x,y)$ is defined as
\begin{equation*}
    \textrm{Stationary Gap}=\| v \| + \| Ax -b \|,
\end{equation*}
where $v\in \nabla f(x)+A^Ty+\mathcal{N}(x)$ and $\mathcal{N}(x)$ is the normal cone of the set $\| x \| \leq c$, defined as
\begin{align*}
\mathcal{N}(x)=\{u\in\mathbb{R}^n\mid \langle u, x^\prime - x \rangle\leq 0 ,\ \| x^\prime\| \leq c\}=\left\{ 
\begin{aligned}
0,\ &\| x\|<c,\\
\tau x,\ &\| x\| =c\ (\tau\geq 0).
\end{aligned}
\right.
\end{align*}
The stationary gap is evaluated by specifying $\tau$ such that $\| v \|$ takes the minimal value. According to \cref{thm:main}, parameters in SProx-ALM are set as: 
$$p=3\sigma_{\max}(Q),\ \gamma=10\frac{\sigma_{\max}(Q)}{\sigma_{\max}^2(A)},\ c=\frac{1}{2(4\sigma_{\max}(Q)+\gamma\sigma_{\max}^2(A))}, \ \alpha=\frac{c\sigma_{\max}^2(Q)}{\sigma_{\max}^2(A)},$$
and different $\beta$ are considered, i.e., $\beta=0.05,0.2,0.5$. The convergence curves of SProx-ALM for different problem size $n$ are compared in Fig.~\cref{fig:SProxALM}, where curves correspond to the median value over $20$ independent trials. Fig.~\cref{fig:SProxALM} shows that increasing $\beta$ improves convergence speed empirically. 


\begin{figure}[H]
    \centering
    \includegraphics[width=0.32\linewidth]{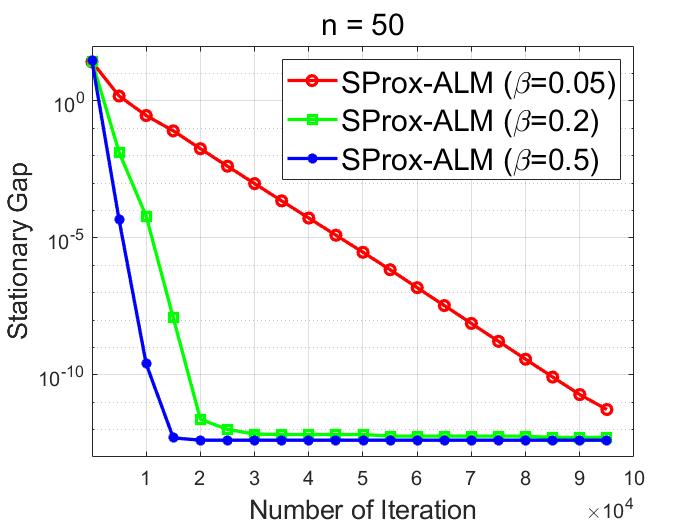}
    \includegraphics[width=0.32\linewidth]{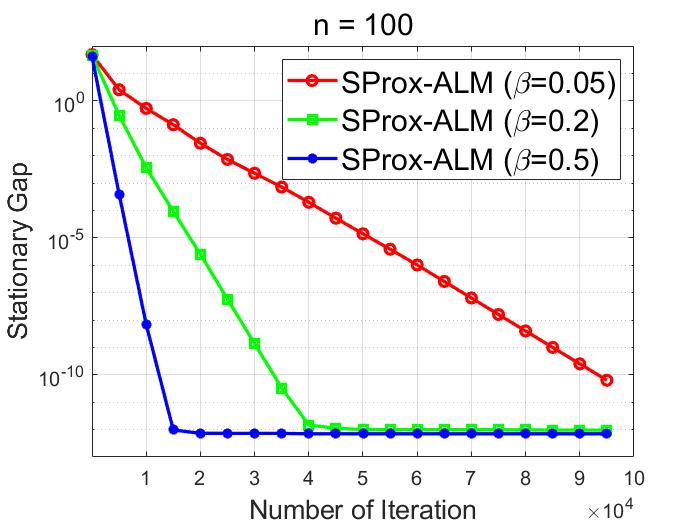}
    \includegraphics[width=0.32\linewidth]{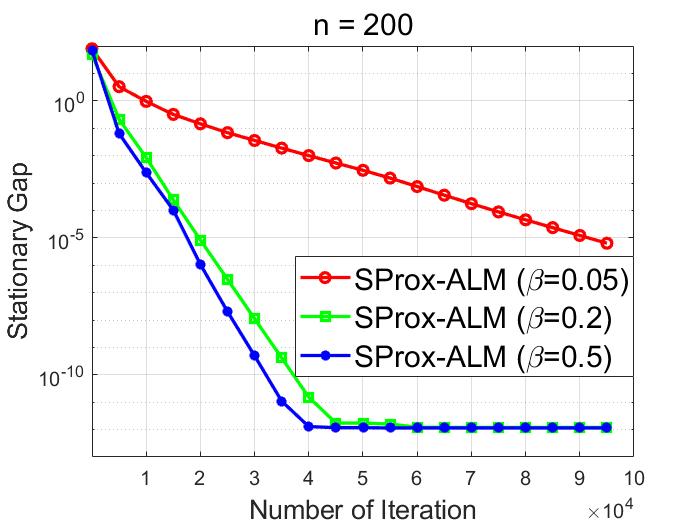}
    \caption{Convergence curves of SProx-ALM with different $\beta$.} 
    \label{fig:SProxALM}
\end{figure}

Next, SProx-ALM ($\beta=0.2$) is compared with two representative algorithms for solving nonconvex QP problem~\eqref{opt:numerQP} with $m=20$. One is the perturbed proximal primal-dual algorithm (PProx-PDA) in~\cite{Hong-perturbed} and the other one is the quadratic penalty accelerated inexact proximal point method (QP-AIPP) in~\cite{Monteiro19}. For PProx-PDA, algorithm parameters (c.f. \cite[Eq. (33) in Page 219]{Hong-perturbed}) are set as:
$$\tau=0.5, \ c=\frac{1}{\tau}-1=1, \ \beta=1.01(3+4c)\sigma_{\max}(Q), \ \rho=\beta,\  \gamma=\tau/(10^3\rho).$$
The coefficient matrix $B^TB$ in PProx-PDA is specified as $B^TB=\sigma I-A^TA$, where $\sigma=1.01\sigma_{\max}(A^TA)$. This makes the primal subproblem in PProx-PDA become a projection problem which admits a closed-form solution. As for QP-AIPP, its algorithm parameters (c.f. \cite[QP-AIPP method in Page 2585]{Monteiro19}) are set as: 
$$
\sigma=0.3, \ m_f=\sigma_{\min}(Q),\ L_f=\sigma_{\max}(Q), \ \hat{\rho}=\hat{\eta}=10^{-4},\ \hat{c} = 1.
$$

The convergence curves for different problem size $n=50,100,200$ are compared in Fig.~\cref{fig:Comp}. One iteration of both SProx-ALM and PProx-PDA corresponds to one step of primal dual update and one iteration of QP-AIPP refers to one step of AIPP update. Fig.~\cref{fig:Comp} shows that SProx-ALM enjoys faster convergence than the other two baseline algorithms. 
\begin{figure}[H]
    \centering
    \includegraphics[width=0.32\linewidth]{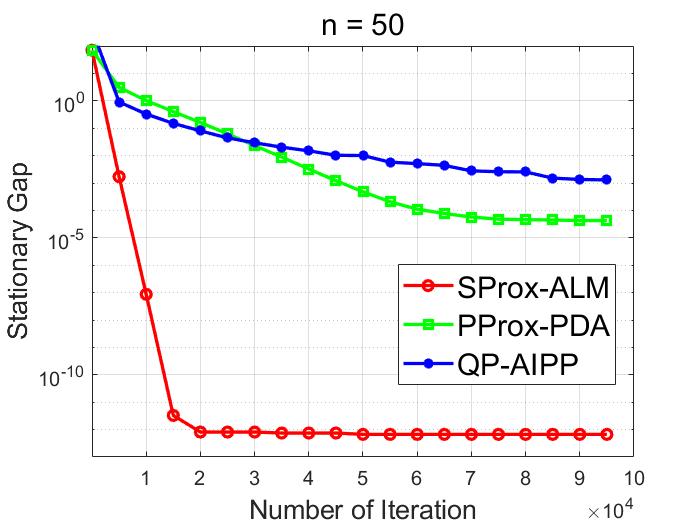}
    \includegraphics[width=0.32\linewidth]{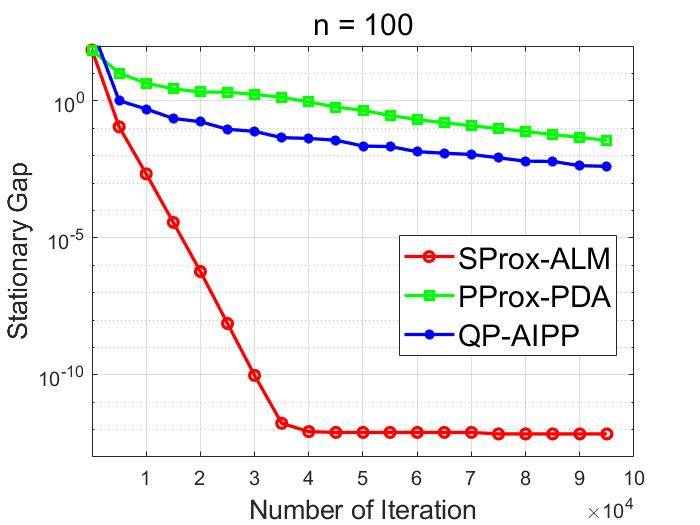}
    \includegraphics[width=0.32\linewidth]{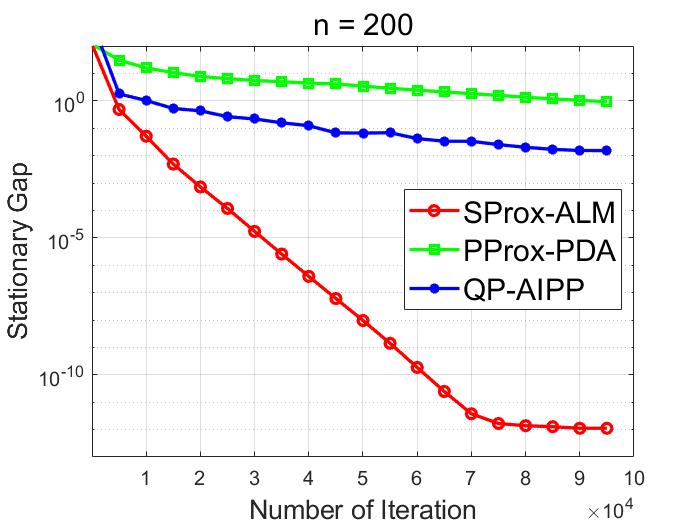}
    \caption{Convergence curves of SProx-ALM, PProx-PDA, and QP-AIPP.} 
    \label{fig:Comp}
\end{figure}

\newpage

\appendix

\section{Proof of \cref{four terms}}\label{app:fourterms}
\begin{proof}
Combining \cref{primal}, \cref{dual ascent}, and \cref{proximal-descent}, we have
  \begin{equation}\label{first descent}
    \begin{aligned}
      &\phi^t-\phi^{t+1}\\
      \ge&  \left(\frac{1}{2c}\|x^{t+1}-x^t\|^2-\alpha\|Ax^t-b\|^2+\frac{p}{2\beta}\|z^t-z^{t+1}\|^2\right)\\
      &+2 \left(\alpha(Ax^t-b)^T(Ax(y^{t+1}, z^t)-b)+p(z^{t+1}-z^t)^T(z^{t+1}+z^t-2x(y^{t+1}, z^{t+1})) \right)\\
      &+2 \left(p(z^{t+1}-z^t)^T({x}(z^t)-z^t)-\frac{p}{2{\sigma}_4}\|z^{t+1}-z^t\|^2 \right)\\
      =& \left(\frac{1}{2c}\|x^{t+1}-x^t\|^2-\alpha\|Ax^t-b\|^2+\frac{p}{2\beta}\|z^t-z^{t+1}\|^2\right)\\
      &+2\alpha(Ax^t-b)^T(Ax(y^{t+1}, z^t)-b)-\frac{p}{{\sigma}_4}\|z^t-z^{t+1}\|^2\\
      &+p(z^{t+1}-z^t)^T\left((z^{t+1}-z^t)-2(x(y^{t+1}, z^{t+1})-{x}(z^t))\right)
      \\
      =& \left(\frac{1}{2c}\|x^{t+1}-x^t\|^2-\alpha\|Ax^t-b\|^2+\frac{p}{2\beta}\|z^t-z^{t+1}\|^2\right)\\
      &+2\alpha(Ax^t-b)^T(Ax(y^{t+1}, z^t)-b)-\frac{p}{{\sigma}_4}\|z^t-z^{t+1}\|^2\\
      &+p(z^{t+1}-z^t)^T\left((z^{t+1}-z^t)-2(x(y^{t+1}, z^{t+1})-x(y^{t+1}, z^t))\right)\\
      &+p(z^{t+1}-z^t)^T\left(2(x(y^{t+1}, z^t)-{x}(z^t))\right)
    \end{aligned}
  \end{equation}
Let $\zeta$ be an arbitrary positive scalar, and by the fact that 
$$\|(z^{t+1}-z^t)/\sqrt{\zeta}+\sqrt{\zeta}(x(y^{t+1}, z^t)-{x}(z^t))\|^2\ge 0,$$ 
we have $$2(z^{t+1}-z^t)^T(x(y^{t+1}, z^t)-{x}(z^t))\ge -\|z^t-z^{t+1}\|^2/\zeta-\zeta\|x(y^{t+1}, z^{t})-{x}(z^t)\|^2.$$
Using Cauchy-Schwarz inequality and the primal error bound \eqref{eb3}, we have
    \begin{eqnarray*}
    -2(z^{t+1}-z^t)^T(x(y^{t+1}, z^{t+1})-x(y^{t+1}, z^t))&\ge& -\|z^t-z^{t+1}\|\|x(y^{t+1}, z^{t+1})-x(y^{t+1}, z^t)\|\\
    &\ge& -\frac{1}{\sigma_4}\|z^t-z^{t+1}\|^2.
    \end{eqnarray*}
Substituting these two inequalities into \eqref{first descent}, we have
    \begin{eqnarray*}
        &&\phi^t-\phi^{t+1}\nonumber\\
        &\ge & \frac{1}{2c}\|x^{t+1}-x^t\|^2-\left(\alpha\|Ax^t-b\|^2-2\alpha(Ax^t-b)^T(Ax(y^{t+1}, z^t)-b)+\alpha\|Ax(y^{t+1}, z^t)-b\|^2\right)\nonumber\\
        &&+\alpha\|Ax(y^{t+1}, z^t)-b\|^2
        +\left(\frac{p}{2\beta}+p-\frac{p}{\sigma_4}-\frac{p}{\zeta}-\frac{p}{{\sigma}_4}\right)\|z^t-z^{t+1}\|^2-p\zeta\|x(y^{t+1}, z^t)-{x}(z^t)\|^2.
    \end{eqnarray*}
By completing the square, we further obtain
    \begin{eqnarray}
    \phi^t-\phi^{t+1}
    &\ge& \frac{1}{2c}\|x^{t+1}-x^t\|^2-\alpha\|A(x(y^{t+1}, z^t)-x^t)\|^2
    +\alpha\|Ax(y^{t+1}, z^t)-b\|^2\nonumber\\
    &&+\left(\frac{p}{2\beta}+p-\frac{2p}{\sigma_4}-\frac{p}{\zeta}\right)\|z^t-z^{t+1}\|^2
    -p\zeta\|x(y^{t+1}, z^t)-{x}(z^t)\|^2\nonumber\\[3pt]
    &\geq
    & \left(\frac{1}{2c}-\frac{\alpha\sigma_{\rm{max}}^2(A)}{c^2(p-L_f)^2}\right)\|x^t-x^{t+1}\|^2+\alpha\|Ax(y^{t+1}, z^t)-b\|^2
    \nonumber\\
    &&
    +\left(\frac{p}{2\beta}+p-\frac{2p}{\sigma_4}-\frac{p}{\zeta}\right)\|z^t-z^{t+1}\|^2
    -p\zeta\|x(y^{t+1}, z^t)-{x}(z^t)\|^2, \label{one}
    \end{eqnarray}
where the second inequality is due to the primal error bound \eqref{eb1} in \cref{error bound}. Next, we use conditions defined in \cref{four terms} to bound the constant terms in \eqref{one}. Setting $\alpha<\frac{c(p-L_f)^2}{4\sigma_{\rm{max}}^2(A)}$, we have $\frac{1}{2c}-\frac{\alpha\sigma_{\rm{max}}^2(A)}{c^2(p-L_f)^2}>\frac{1}{4c}$. Further, setting $p\geq 3 L_f$, we have $\sigma_4=(p-L_f)/p\geq \frac{2}{3}$. This together with $\zeta=12 \beta$ and $\beta<\frac{1}{24}$ implies
\begin{equation*}
\begin{aligned}
(\frac{p}{2\beta}+p-\frac{2p}{\sigma_4}-\frac{p}{12\beta})>(\frac{5p}{12\beta}-2p)\geq\frac{p}{3\beta}.
\end{aligned}
\end{equation*}

Finally, combining with \eqref{one}, we have
\begin{equation*}
\begin{aligned}
	\phi^t-\phi^{t+1}\ge&\frac{1}{4c}\|x^t-x^{t+1}\|^2+\alpha\|Ax(y^{t+1}, z^t)-b\|^2\\
	\quad &+\frac{p}{3\beta}\|z^t-z^{t+1}\|^2-12p\beta\|x(y^{t+1}, z^t)-{x}(z^t)\|^2. 
\end{aligned}
\end{equation*}
This completes the proof.
\end{proof}

\section{Proof of \cref{suff-decrease}}\label{app:suff-decrease}
\begin{proof}
    Let $\delta$ be the constant in \cref{local-eb} and $\Delta$ be the constant in \cref{prop:slater}. Denote $M=\max_{x,x^\prime\in \mathcal{X}}\| x - x^\prime\|$ and define the following three conditions: 
    \begin{eqnarray}
        \| x^t - x^{t+1} \|^2&\leq& 96cpM^2\beta ,\label{eq:cond_x}\\
        \|Ax(y^{t+1}, z^t)-b\|^2 & \leq&\frac{24pM^2}{\alpha}\beta,\label{eq:cond_Ax}\\
        \|z^t-x^{t+1}\|^2 &\leq& 72\|x(y^{t+1}, z^t)-{x}(z^t)\|^2, \label{eq:cond_xz}
    \end{eqnarray}
    where $\beta<\beta^\prime=\frac{1}{24}\min\{1,\frac{\eta^2}{8cpM^2},\frac{\alpha\eta^2}{pM^2},\frac{\alpha}{p\sigma_s^2}  \}$ with  $\eta=\min\{ \delta,\Delta/2,\frac{\delta}{72\sigma_w} \}$.
    \begin{enumerate}[fullwidth,itemindent=0em]
        \item \textbf{Case 1:} Suppose \eqref{eq:cond_x}, \eqref{eq:cond_Ax}, and \eqref{eq:cond_xz} hold. Then, for $\beta<\beta^\prime$, we have 
        \begin{align}
            \| x^t - x^{t+1} \|& \leq \sqrt{96cpM^2\beta}\leq\eta\leq \delta,\label{eq:case1_x}\\
            \|Ax(y^{t+1}, z^t)-b\| &\leq \sqrt{\frac{24pM^2}{\alpha}\beta} \leq\eta\leq \delta. \label{eq:case1_Ax}
        \end{align}
        Since $\eta=\min\{ \delta,\Delta/2,\frac{\delta^2}{72\sigma_w} \}\leq \Delta/2$, by \cref{weak-eb-ite}, \eqref{eq:case1_Ax} implies 
        $$\|x(y^{t+1}, z^t)-{x}(z^t)\|^2\le \sigma_w\|Ax(y^{t+1}, z^t)-b\|.$$
        Combine this with condition~\eqref{eq:cond_xz}, we have 
        \begin{align}\label{eq:case1_xz}
            \|z^t-x^{t+1}\|^2&\leq 72\|x(y^{t+1}, z^t)-{x}(z^t)\|^2\notag\\
            &\leq 72\sigma_w\|Ax(y^{t+1}, z^t)-b\|\notag\\
            &\leq 72\sigma_w \eta \leq \delta^2.
        \end{align}
        Notice that \eqref{eq:case1_x}, \eqref{eq:case1_Ax}, and \eqref{eq:case1_xz} are conditions in \cref{local-eb}, hence we have 
        {\small  
        \begin{align*}
            &\alpha\|Ax(y^{t+1}, z^t)-b\|^2-12p\beta\|x(y^{t+1}, z^t)-{x}(z^t)\|^2\nonumber\\
            &=\frac{\alpha}{2}\|Ax(y^{t+1}, z^t)-b\|^2+\frac{\alpha}{2}\|Ax(y^{t+1}, z^t)-b\|^2-12p\beta\|x(y^{t+1}, z^t)-{x}(z^t)\|^2\nonumber\\
            &\geq\frac{\alpha}{2}\|Ax(y^{t+1}, z^t)-b\|^2+\frac{\alpha}{2}\|Ax(y^{t+1}, z^t)-b\|^2-12p\beta\sigma_s^2\|Ax(y^{t+1}, z^t)-b\|^2\nonumber\\
            &\geq\frac{\alpha}{2}\|Ax(y^{t+1}, z^t)-b\|^2,
        \end{align*}
        }
        where the first inequality is by \cref{local-eb} and the second one is due $\beta<\frac{\alpha}{24p\sigma_s^2}$.
        
        \item \textbf{Case 2:} Suppose one of the conditions~\eqref{eq:cond_x}-\eqref{eq:cond_xz} is violated, then we consider the following 3 subcases.
        \begin{enumerate}[fullwidth,itemindent=0em]
            \item Case 2.1: Suppose $\|x^t-x^{t+1}\|^2> 96cpM^2\beta$, then 
             {\small 
            \begin{align*}
                &\frac{1}{4c}\|x^t-x^{t+1}\|^2-12p\beta\|x(y^{t+1}, z^t)-{x}(z^t)\|^2\nonumber\\
                &=\frac{1}{8c}\|x^t-x^{t+1}\|^2+\frac{1}{8c}\|x^t-x^{t+1}\|^2-12p\beta\|x(y^{t+1}, z^t)-{x}(z^t)\|^2\nonumber\\
                &\geq\frac{1}{8c}\|x^t-x^{t+1}\|^2+\frac{1}{8c}96cpM^2\beta-12p\beta M^2=\frac{1}{8c}\|x^t-x^{t+1}\|^2
            \end{align*}
            }
            where the inequality is due $\|x^t-x^{t+1}\|^2> 96cpM^2\beta$.
            
            \item Case 2.2: Suppose $\|Ax(y^{t+1}, z^t)-b\|^2>\frac{24pM^2}{\alpha}\beta$, then 
            {\small 
            \begin{align*}
                &\alpha\|Ax(y^{t+1}, z^t)-b\|^2-12p\beta\|x(y^{t+1}, z^t)-{x}(z^t)\|^2\nonumber\\
                &=\frac{\alpha}{2}\|Ax(y^{t+1}, z^t)-b\|^2+\frac{\alpha}{2}\|Ax(y^{t+1}, z^t)-b\|^2-12p\beta\|x(y^{t+1}, z^t)-{x}(z^t)\|^2\nonumber\\
                &\geq\frac{\alpha}{2}\|Ax(y^{t+1}, z^t)-b\|^2+\frac{\alpha}{2}\frac{24pM^2}{\alpha}\beta-12p\beta M^2\\
                &=\frac{\alpha}{2}\|Ax(y^{t+1}, z^t)-b\|^2
            \end{align*}
            }
            where the inequality is due $\|Ax(y^{t+1}, z^t)-b\|^2>\frac{24pM^2}{\alpha}\beta$.
            
            \item Case 2.3: Suppose $\|z^t-x^{t+1}\|^2>72\|x(y^{t+1}, z^t)-{x}(z^t)\|^2$, then  {\small 
            \begin{align*}
                &\frac{p}{3\beta}\|z^t-z^{t+1}\|^2-12p\beta\|x(y^{t+1}, z^t)-{x}(z^t)\|^2\nonumber\\
                &=\frac{p}{6\beta}\|z^t-z^{t+1}\|^2+\frac{p}{6\beta}\|z^t-z^{t+1}\|^2-12p\beta\|x(y^{t+1}, z^t)-{x}(z^t)\|^2\\
                &=\frac{p}{6\beta}\|z^t-z^{t+1}\|^2+\frac{p}{6\beta}\beta^2\|z^t-x^{t+1}\|^2-12p\beta\|x(y^{t+1}, z^t)-{x}(z^t)\|^2\nonumber\\
                &\geq\frac{p}{6\beta}\|z^t-z^{t+1}\|^2
            \end{align*}
            }
            where second equality is due fact $\beta\|z^t-x^{t+1}\|=\|z^t-z^{t+1}\|$ and the inequality is because $\|z^t-x^{t+1}\|>72\|x(y^{t+1}, z^t)-{x}(z^t)\|^2$.
        \end{enumerate}
    \end{enumerate}
\end{proof}

\section{Proof of \cref{local-eb2}}\label{subsec:perb_bound}

Useful technical lemmas are firstly given. Define $\mathcal{Y}^*(r)$ to be the dual solution set of the KKT conditions in \eqref{eq:KKT_xrz}. Then similar to the weak error bound given in \cref{weak-eb-ite}, we have the following lemma.
\begin{lem}\label{lem:weak-eb}
For any $r, r^\prime\in\mathbb{R}^m$ with  $\|r\|, \|r^\prime\|\leq \Delta$, we have 
\begin{equation}\label{weak2}
\|\hat{x}(r, z)-\hat{x}(r^\prime, z)\|^2<\zeta \langle r-r^\prime, y-y^\prime\rangle,
\end{equation}
where $y\in \mathcal{Y}^*(r), y^\prime\in \mathcal{Y}^*(r^\prime)$, $z\in\mathcal{X}$, and $\zeta=\frac{1}{2(p-L_f)}$.
\end{lem}
The proof of \cref{lem:weak-eb} is presented in \cref{app:weak-eb}.

To prove \cref{local-eb2},
we need to bound the term  
$\|\hat{x}(r, z)-\hat{x}(r', z)\|$
by $\|r-r'\|$. Intuitively, we hope $y-y'$ can not dominate the  right-hand-side of \eqref{weak2}. 
Therefore, we need to bound the term $y-y^\prime$ by $\hat{x}(r, z)-\hat{x}(r^\prime, z)$. 
We will prove that this kind of `inverse bound' holds locally, where the meaning of `locally' is described by using the following definition.

\begin{definition}[Basic Set]\label{def:basicset}
An index set $\mathcal{S}\subseteq [m+\ell]$ is said to be a basic set of $(r, z)$ if $Q_{\mathcal{S}}(\hat{x}(r, z))$ is of full row 
rank and 
there exist $y\in y^0+\mathrm{range}(A)$ and $\mu$ satisfying the KKT conditions \eqref{eq:KKT_xrz} with $y_i=0$ and $\mu_j=0$, where $i\in[m]$, $j\in [m+1,m+\ell]$, and $i,j \notin \mathcal{S}$.
Moreover, we say that $(r, z)$ and  $(r^\prime, z)$ share a common basic set if there exists a $\mathcal{S}\subset [m+\ell]$ such 
that $
\mathcal{S}$ is a basic  set of both $(r, z)$ and  $(r^\prime, z)$.
\end{definition}

The basic set defined in \cref{def:basicset} has the following existence property. 
\begin{lem}\label{basic-set}
For any $(r, z)\in \Omega$, $r$ has at least one basic set.
\end{lem}

The proof of \cref{basic-set} is presented in \cref{app:basic-set}. Based on \cref{def:basicset}, we have the following `inverse bound'. 
\begin{lem}\label{lem:inverse}
If $(r, z)\in \Omega$ and $(r^\prime, z)\in \Omega$ share a common basic set $\mathcal{S}$, then there exists a constant $\kappa>0$ such that $\mathrm{dist}(\mathcal{Y}^*(r), \mathcal{Y}^*(r^\prime))\le \kappa\|\hat{x}(r, z)-\hat{x}(r^\prime, z)\|.$
\end{lem}

The proof of \cref{lem:inverse} is presented in \cref{app:lem:inverse}. Combining \cref{lem:weak-eb} and~\cref{lem:inverse} and using Cauchy-Schwartz inequality, we immediately have the following corollary.
\begin{coro}\label{locallocal-eb}
If $(r, z)\in \Omega$ and $(r^\prime, z)\in \Omega$ share a common basic set $\mathcal{S}$, then there exists a constant $\sigma_s=\kappa\zeta>0$ such that $\|\hat{x}(r, z)-\hat{x}(r^\prime, z)\|<\sigma_s\|r-r^\prime\|.$
\end{coro}
\cref{locallocal-eb} states that for any $(r, z)$ and $(r', z)$ which share a common basic set, we have the dual error bound. If $(r, z)$ shares a common basic set with $(r^\prime, z)=(0, z)$, then \cref{local-eb2} can be implied by this proposition. However, it is not straightforward to prove the existence of a common basic set between $(r, z)$ and $(0, z)$. To tackle this difficulty, we establish 
the following decomposition proposition. 
\begin{prop}\label{dual-pieces1}
For any $({r}, {z})\in \Omega$, there exists a finite sequence $0=\lambda_0<\lambda_1<\cdots<\lambda_R=1$ such that $
\lambda_i{r}$ shares a common basic set with $\lambda_{i+1}{r}$.
\end{prop}
The proof of \cref{dual-pieces1} is presented in \cref{app:propdec}. Next, we use \cref{locallocal-eb} and \cref{dual-pieces1} to prove \cref{local-eb2}.

\medskip
\noindent \textbf{Proof of \cref{local-eb2}:} 
By \cref{locallocal-eb} and \cref{dual-pieces1}, we have
\begin{equation}\label{local}
\|\hat{x}(\lambda_i r, z)-\hat{x}(\lambda_{i+1}r, z)\|\le \sigma_s\|\lambda_i r-\lambda_{i+1} r\|.
\end{equation}
Summing up \eqref{local} from $0$ to $R-1$ and using the triangular inequality, we have
\begin{eqnarray*}
\sum_{i=0}^{R-1}\|\hat{x}(\lambda_i r, z)-\hat{x}(\lambda_{i+1} r, z)\|&\le&\sum_{i=0}^{R-1}\sigma_s\|\lambda_i r-\lambda_
{i+1} r\|\\
&=& \sigma_s\| \tilde{r} \|,
\end{eqnarray*}
where the last equality is because $0=\lambda_0<\lambda_1<\cdots<\lambda_R=1$. This completes the proof.

\newpage 
\bibliographystyle{siamplain}
\bibliography{references}

\begin{thebibliography}{10}

\bibitem{Bertsekas}
{\sc D.~P. Bertsekas}, {\em Nonlinear programming}, Journal of the Operational
  Research Society, 48 (1997), pp.~334--334.

\bibitem{paral}
{\sc D.~P. Bertsekas and J.~N. Tsitsiklis}, {\em Parallel and distributed
  computation: numerical methods}, vol.~23, Prentice hall Englewood Cliffs, NJ,
  1989.

\bibitem{carmon2020lower}
{\sc Y.~Carmon, J.~C. Duchi, O.~Hinder, and A.~Sidford}, {\em Lower bounds for
  finding stationary points i}, Mathematical Programming, 184 (2020),
  pp.~71--120.

\bibitem{svm}
{\sc R.~De~Leone and C.~Lazzari}, {\em Error bounds for support vector machines
  with application to the identification of active constraints}, Optim. Methods
  \& Software, 25 (2010), pp.~185--202.

\bibitem{Deng2016}
{\sc W.~Deng and W.~Yin}, {\em On the global and linear convergence of the
  generalized alternating direction method of multipliers}, Journal of
  Scientific Computing, 66 (2016), pp.~889--916.

\bibitem{eckstein-bertsekas}
{\sc J.~Eckstein and D.~P. Bertsekas}, {\em On the douglas rachford splitting
  method and the proximal point algorithm for maximal monotone operators},
  Mathematical Programming, 55 (1992), pp.~293--318.

\bibitem{facchinei2007finite}
{\sc F.~Facchinei and J.-S. Pang}, {\em Finite-dimensional variational
  inequalities and complementarity problems}, Springer Science \& Business
  Media, 2007.

\bibitem{Gao-Goldfarb-Curtis-2018}
{\sc W.~Gao, D.~Goldfarb, and F.~E. Curtis}, {\em {ADMM }for multiaffine
  constrained optimization}, Optimization Methods and Software,  (2019),
  pp.~1--47.

\bibitem{Hong-perturbed}
{\sc D.~Hajinezhad and M.~Hong}, {\em Perturbed proximal primal--dual algorithm
  for nonconvex nonsmooth optimization}, Mathematical Programming, 176 (2019),
  pp.~207--245.

\bibitem{He2012}
{\sc B.~He, M.~Tao, and X.~Yuan}, {\em Alternating direction method with
  gaussian back substitution for separable convex programming}, SIAM Journal on
  Optimization, 22 (2012), pp.~313--340.

\bibitem{Hoffmanbound}
{\sc A.~J. Hoffman}, {\em On approximate solutions of systems of linear
  inequalities}, in Selected Papers Of Alan J Hoffman: With Commentary, World
  Scientific, 2003, pp.~174--176.

\bibitem{Hong17}
{\sc M.~Hong, D.~Hajinezhad, and M.~M. Zhao}, {\em {Prox-PDA}: The proximal
  primal-dual algorithm for fast distributed nonconvex optimization and
  learning over networks}, in Proceedings of the 34th International Conference
  on Machine Learning, 2017.

\bibitem{Luo-Hong12}
{\sc M.~Hong and Z.-Q. Luo}, {\em On the linear convergence of the alternating
  direction method of multipliers}, Mathematical Programming, 162 (2017),
  pp.~165--199.

\bibitem{Hong-Luo16}
{\sc M.~Hong, Z.-Q. Luo, and M.~Razaviyayn}, {\em Convergence analysis of
  alternating direction method of multipliers for a family of nonconvex
  problems}, SIAM Journal on Optimization, 26 (2016), pp.~337--364.

\bibitem{janin1984directional}
{\sc R.~Janin}, {\em Directional derivative of the marginal function in
  nonlinear programming}, in Sensitivity, Stability and Parametric Analysis,
  Springer, 1984, pp.~110--126.

\bibitem{JiangMa16}
{\sc B.~Jiang, T.~Lin, S.~Ma, and S.~Zhang}, {\em Structured nonconvex and
  nonsmooth optimization: algorithms and iteration complexity analysis},
  Computational Optimization and Applications, 72 (2019), pp.~115--157.

\bibitem{Monteiro19}
{\sc W.~Kong, J.~G. Melo, and R.~D. Monteiro}, {\em Complexity of a quadratic
  penalty accelerated inexact proximal point method for solving linearly
  constrained nonconvex composite programs}, SIAM Journal on Optimization, 29
  (2019), pp.~2566--2593.

\bibitem{kong2020iteration}
{\sc W.~Kong, J.~G. Melo, and R.~D. Monteiro}, {\em Iteration-complexity of a
  proximal augmented lagrangian method for solving nonconvex composite
  optimization problems with nonlinear convex constraints}, arXiv preprint
  arXiv:2008.07080,  (2020).

\bibitem{GLI15}
{\sc G.~Li and T.~K. Pong}, {\em Global convergence of splitting methods for
  nonconvex composite optimization}, SIAM Journal on Optimization, 25 (2015),
  pp.~2434--2460, \url{https://doi.org/10.1137/140998135}.

\bibitem{robustleanring}
{\sc J.~Li, S.~Huang, and A.~M.-C. So}, {\em A first-order algorithmic
  framework for wasserstein distributionally robust logistic regression}, arXiv
  preprint arXiv:1910.12778,  (2019).

\bibitem{lin2022complexity}
{\sc Q.~Lin, R.~Ma, and Y.~Xu}, {\em Complexity of an inexact proximal-point
  penalty method for constrained smooth non-convex optimization}, Computational
  Optimization and Applications, 82 (2022), pp.~175--224.

\bibitem{luo-analytic}
{\sc Z.-Q. Luo and J.-S. Pang}, {\em Error bounds for analytic systems and
  their applications}, Math. Prog., 67 (1994), pp.~1--28.

\bibitem{luo-Hoffman}
{\sc Z.-Q. Luo and J.~F. Sturm}, {\em Error bounds for quadratic systems}, in
  High performance optimization, Springer, 2000, pp.~383--404.

\bibitem{luo-linear}
{\sc Z.-Q. Luo and P.~Tseng}, {\em On a global error bound for a class of
  monotone affine variational inequality problems}, Oper. Res. Letters, 11
  (1992), pp.~159--165.

\bibitem{bcd}
{\sc Z.-Q. Luo and P.~Tseng}, {\em On the convergence of the coordinate descent
  method for convex differentiable minimization}, Journal of Optimization
  Theory and Applications, 72 (1992), pp.~7--35.

\bibitem{composite}
{\sc Z.-Q. Luo and P.~Tseng}, {\em On the linear convergence of descent methods
  for convex essentially smooth minimization}, SIAM J. Control Optim., 30
  (1992), pp.~408--425.

\bibitem{eb-survey}
{\sc Z.-Q. Luo and P.~Tseng}, {\em Error bounds and convergence analysis of
  feasible descent methods: a general approach}, Ann. Oper. Res., 46 (1993),
  pp.~157--178.

\bibitem{luo-dual}
{\sc Z.-Q. Luo and P.~Tseng}, {\em On the convergence rate of dual ascent
  methods for linearly constrained convex minimization}, Math. Oper. Res., 18
  (1993), pp.~846--867.

\bibitem{gebkkt}
{\sc O.~L. Mangasarian and R.~De~Leone}, {\em Error bounds for strongly convex
  programs and (super) linearly convergent iterative schemes for the least
  2-norm solution of linear programs}, Appl. Math. Optim., 17 (1988),
  pp.~1--14.

\bibitem{pang1987posteriori}
{\sc J.-S. Pang}, {\em A posteriori error bounds for the linearly-constrained
  variational inequality problem}, Math. Oper. Res., 12 (1987), pp.~474--484.

\bibitem{Pang-survey}
{\sc J.-S. Pang}, {\em Error bounds in mathematical programming}, Math. Prog.,
  79 (1997), pp.~299--332.

\bibitem{parikh2014proximal}
{\sc N.~Parikh and S.~Boyd}, {\em Proximal algorithms}, Foundations and Trends
  in optimization, 1 (2014), pp.~127--239.

\bibitem{cvxana}
{\sc R.~T. Rockafellar}, {\em Convex analysis}, vol.~28, Princeton university
  press, 1970.

\bibitem{matrixsep}
{\sc Y.~Shen, Z.~Wen, and Y.~Zhang}, {\em Augmented lagrangian alternating
  direction method for matrix separation based on low-rank factorization},
  Optimization Methods and Software, 29 (2014), pp.~239--263.

\bibitem{solodov2010constraint}
{\sc M.~V. Solodov et~al.}, {\em Constraint qualifications}, Wiley Encyclopedia
  of Operations Research and Management Science. Wiley, New York,  (2010).

\bibitem{sun2021algorithms}
{\sc K.~Sun and X.~A. Sun}, {\em Algorithms for difference-of-convex (dc)
  programs based on difference-of-moreau-envelopes smoothing}, arXiv e-prints,
  (2021), pp.~arXiv--2104.

\bibitem{wang2021linear}
{\sc P.~Wang, H.~Liu, and A.~M.-C. So}, {\em Linear convergence of a proximal
  alternating minimization method with extrapolation for $\ell_1 $-norm
  principal component analysis}, arXiv preprint arXiv:2107.07107,  (2021).

\bibitem{wang2021non}
{\sc P.~Wang, Z.~Zhou, and A.~M.-C. So}, {\em Non-convex exact community
  recovery in stochastic block model}, Mathematical Programming,  (2021),
  pp.~1--37.

\bibitem{wyin16}
{\sc Y.~Wang, W.~Yin, and J.~Zeng}, {\em Global convergence of admm in
  nonconvex nonsmooth optimization}, Journal of Scientific Computing, 78
  (2019), pp.~29--63.

\bibitem{phase}
{\sc Z.~Wen, C.~Yang, X.~Liu, and S.~Marchesini}, {\em Alternating direction
  methods for classical and ptychographic phase retrieval}, Inverse Problems,
  28 (2012), p.~115010.

\bibitem{yan2020collaborative}
{\sc J.~Yan, W.~Pu, S.~Zhou, H.~Liu, and Z.~Bao}, {\em Collaborative detection
  and power allocation framework for target tracking in multiple radar system},
  Information Fusion, 55 (2020), pp.~173--183.

\bibitem{imageback}
{\sc L.~Yang, T.~K. Pong, and X.~Chen}, {\em Alternating direction method of
  multipliers for a class of nonconvex and nonsmooth problems with applications
  to background/foreground extraction}, SIAM Journal on Imaging Sciences, 10
  (2017), pp.~74--110.

\bibitem{So19}
{\sc M.-C. Yue, Z.~Zhou, and A.~M.-C. So}, {\em A family of inexact sqa methods
  for non-smooth convex minimization with provable convergence guarantees based
  on the luo--tseng error bound property}, Math. Prog., 174 (2019),
  pp.~327--358.

\bibitem{zeng2021moreau}
{\sc J.~Zeng, W.~Yin, and D.-X. Zhou}, {\em Moreau envelope augmented
  lagrangian method for nonconvex optimization with linear constraints}, arXiv
  preprint arXiv:2101.08519,  (2021).

\bibitem{zhang2018proximal}
{\sc J.~Zhang and Z.-Q. Luo}, {\em A proximal alternating direction method of
  multiplier for linearly constrained nonconvex minimization}, SIAM Journal on
  Optimization, 30 (2020), pp.~2272--2302.

\bibitem{zhang2020global}
{\sc J.~Zhang and Z.-Q. Luo}, {\em A global dual error bound and its
  application to the analysis of linearly constrained nonconvex optimization},
  SIAM Journal on Optimization, 32 (2022), pp.~2319--2346.

\bibitem{zhou2017unified}
{\sc Z.~Zhou and A.~M.-C. So}, {\em A unified approach to error bounds for
  structured convex optimization problems}, Mathematical Programming, 165
  (2017), pp.~689--728.

\end{thebibliography}

\newpage 

\section{Proof of the Three Descent Lemmas}\label{app:3descent}
\subsection{Proof of \cref{primal}}
\begin{proof}
First, by $y^{t+1}=y^{t}+\alpha(Ax^t-b)$ we have the trivial equality:
\begin{equation}\label{eq:descent_Ax}
K(x^t, z^t; y^t)-K(x^t, z^t; y^{t+1})
=-\alpha\|Ax^t-b\|^2.
\end{equation}
Notice that updating $x$ is a standard gradient projection step and $K(\cdot,y,z)$ is a strongly convex function. Hence, choosing $c\leq 1/(L_f+\gamma\sigma_{\rm max}^2(A)+p)$, we have 
\begin{equation}\label{eq:descent}
K(x^t, z^t; y^{t+1})-K(x^{t+1}, z^t; y^{t+1})
\ge \frac{1}{2c}\|x^t-x^{t+1}\|^2.
\end{equation}
Moreover, recall $z^{t+1}=z^t+\beta(x^{t+1}-z^t)$, we have
\begin{eqnarray}\label{eq:descent_zz}
K(x^{t+1}, z^{t}; y^{t+1})-K(x^{t+1}, z^{t+1}; y^{t+1}) &=& \frac{p}{2}(\|x^{t+1}-z^t\|^2-\|x^{t+1}-z^{t+1}\|^2) \nonumber\\
&=& \frac{p}{2}(z^{t+1}-z^{t})^T((x^{t+1}-z^t)+(x^{t+1}-z^{t+1})) \nonumber\\
&=& \frac{p}{2}(2/\beta-1)\|z^t-z^{t+1}\|^2 \nonumber\\
&\ge&\frac{p}{2\beta}\|z^t-z^{t+1}\|^2,\label{zprimal}
\end{eqnarray}
where the last inequality is due $\beta\leq 1$. Combining inequalities~\eqref{eq:descent_Ax}, \eqref{eq:descent}, and~\eqref{eq:descent_zz} completes the proof.
\end{proof}

\subsection{Proof of \cref{dual ascent}}
\begin{proof}
First, recall the definition of $d(y,z)$ and $x(y,z)$, we have
\begin{eqnarray}\label{eq:descent_dxy1}
d(y^{t+1}, z^t)-d(y^{t}, z^t)
&=& K(x(y^{t+1}, z^t), z^t; y^{t+1})-K(x(y^{t}, z^t), z^t; y^t)\nonumber\\
&\ge &K(x(y^{t+1}, z^t), z^t; y^{t+1})-K(x(y^{t+1}, z^t), z^t; y^t)\nonumber\\
&=& \langle y^{t+1}-y^t,Ax(y^{t+1}, z^t)-b\rangle,\nonumber\\
&=& \alpha(Ax^t-b)^T(Ax(y^{t+1}, z^t)-b),
\end{eqnarray}
where the inequality is because $x(y^t, z^t)$ minimizes $K(x, z^t; y^t)$.
Next, using the same technique, we have
\begin{eqnarray}\label{eq:descent_dxy2}
d(y^{t+1},z^{t+1})-d(y^{t+1}, z^{t})&=& K(x(y^{t+1}, z^{t+1}), z^{t+1}; y^{t+1})-K(x(y^{t+1}, z^t), z^t; y^{t+1})\nonumber\\
&\ge& K(x(y^{t+1}, z^{t+1}), z^{t+1}; y^{t+1})-K(x(y^{t+1}, z^{t+1}), z^t; y^{t+1})\nonumber\\
&=& \frac{p}{2}(\|x(y^{t+1}, z^{t+1})-z^{t+1}\|^2-\|x(y^{t+1}, z^{t+1})-z^t\|^2)\nonumber\\
&=&\frac{p}{2}(z^{t+1}-z^t)^T(z^{t+1}+z^t-2x(y^{t+1}, z^{t+1})). \label{zdual}
\end{eqnarray}
Combining inequalities \eqref{eq:descent_dxy1} and \eqref{eq:descent_dxy2} completes the proof.
\end{proof}

\subsection{Proof of \cref{proximal-descent}}
\begin{proof}
First, using Danskin's theorem in convex analysis, we have $\nabla P(z^t)=p(z^t-x(z^t)),$ where $P(z)$ is defined in~\eqref{eq:subopt_xz}.
So it suffices to prove $\nabla P(z^t)$ has a Lipschiz constant that
$$\|\nabla P(z^t)-\nabla P(z^{t+1})\|\le p({\sigma}^{-1}_4+1)\|z^t-z^{t+1}\|.$$
But this is a direct corollary of the error bound \eqref{eb4} in \cref{error bound}. This completes the proof.
\end{proof}

\section{Proof of Technical Lemmas in \cref{subsec:perb_bound}}
\subsection{Proof of \cref{lem:weak-eb}}\label{app:weak-eb}
\begin{proof}
  Since $Ax(y, z)-b=r$ and $Ax(y^\prime, z)-b=r^\prime$, by \cref{lem:xyz=xrz}, we have  $\hat{x}(r, z)=x(y, z),\quad \hat{x}(r^\prime, z)=x(y^\prime, z).$
  By the strong convexity of $K(\cdot, z;y)$, we have
  \begin{equation}\label{1.1}
      \begin{aligned}
  &(p-L_f)\|x(y, z)-x(y^\prime, z)\|^2\le K(x(y^\prime, z), z;y)-K(x(y, z), z;y),\\
  &(p-L_f)\|x(y, z)-x(y^\prime, z)\|^2\le K(x(y, z), z;y^\prime)-K(x(y^\prime, z), z;y^\prime).
      \end{aligned}
  \end{equation}
  On the other hand, by the  definition of $K(\cdot)$, we have
  \begin{equation}\label{1.2}
      \begin{aligned}
  &K(x(y, z), z;y)-K(x(y, z), z;y^\prime)=\langle y-y^\prime, Ax(y, z)-b\rangle,\\
  &K(x(y^\prime, z), z;y^\prime)-K(x(y^\prime, z), z;y)=\langle y^\prime-y, Ax(y^\prime, z)-b\rangle.
      \end{aligned}
  \end{equation}
  Combining \eqref{1.1} and \eqref{1.2}, we have
  \begin{eqnarray*}
  2(p-L_f)\|x(y, z)-x(y^\prime, z)\|^2&\le&\langle y-y^\prime, (Ax(y, z)-b)-(Ax(y^\prime, z)-b)\rangle\\
  &=&\langle y-y^\prime, r-r^\prime\rangle.
  \end{eqnarray*}
  Setting $\zeta=\frac{1}{2(p-L_f)}$ completes the proof.
  \end{proof}

\subsection{Proof of \cref{basic-set}}\label{app:basic-set}
We need the following lemma.
\begin{lemma}
Let matrix $M\in \mathbb{R}^{n\times(m+\ell)}$ and vector $v=(v_1,v_2,\ldots, v_{m+\ell})\in \mathbb{R}^{m+\ell}$. Suppose $Mv\geq 0$ and $v_{m+1}, \cdots, v_{m+
\ell}\ge 0$. Then there exists some $v^*\in \mathbb{R}^{m+\ell}$ such that $Mv=Mv^*$ and the columns of $M$ corresponding to index set $\mathcal{S}_{v^*}=\{i\in [m+\ell]\mid v^*_i\ne 0\}$ are linearly independent.
\end{lemma}
\begin{proof}
We prove it by contradiction. Suppose there is no such a $v^*$. Let $v'\in \mathbb{R}^{m+\ell}$ be a vector satisfying 
$v'\in\{u\mid Mu=Mv, u_i\ge 0, i=m+1, \cdots, m+\ell\}$. 
Denote $\mathcal{S}=\mathcal{S}_{v'}$, then by contradiction, the columns of $M$ corresponding to the set $\mathcal{S}$ are linear dependent. Let $w\in 
\mathbb{R}^{m+\ell}$ be a vector satisfying $Mw=0$ and $w_i=0, i\notin \mathcal{S}$.
Denote $q=\min_{i\in \mathcal{S}, w_i\ne 0}\frac{|v^{\prime}_i|}{|w_i|}$,  $i(q)=\arg\min_{i\in \mathcal{S}}\frac{|v^{\prime}_i|}{|w_i|}$, and $\tilde{v}=v^{\prime}-q(w_{i(q)})w$, then $\tilde{v}_{i(q)}=0$ and $M\tilde{v}=Mv$. 
Moreover, by the definition of $q$, we have $\tilde{v}_i\ge 0, i\in \{m+1, \cdots, m+\ell\}\cap \mathcal{S}$, which is a contradiction to the definition of $v^{\prime}$. This completes the proof.
\end{proof}
Then, we can prove \cref{basic-set} by letting $M=Q^T(\hat{x}(r, z))$ and $v=(y, \mu)$. 

\medskip
\noindent \textbf{Proof of \cref{basic-set}:} First, for any $(r, z)\in \Omega$, there exists some $y, \mu$ such that
$$\nabla_x(g(\hat{x}(r, z))+\frac{p}{2}\|\hat{x}(r, z)-z\|^2)+Q^T(\hat{x}(r, z))v=0, $$
where $v=(y, \mu)$.
Then, there exists some $(y^*, \mu^*)$  such that 
$$\nabla_x(g(\hat{x}(r, z))+\frac{p}{2}\|\hat{x}(r, z)-z\|^2)+Q^T(\hat{x}(r, z))v^*=0, $$
where $v^*=(y^*, \mu^*)$ and the columns of $Q^T(\hat{x}(r, z))$ in $\mathcal{S}_{v^*}$ are linearly independent.
Hence, $\mathcal{S}_{v^*}$ is a basic set. This completes the proof.

\subsection{Proof of \cref{lem:inverse}}\label{app:lem:inverse}
Technical lemmas are given as below. We first show that the constraint matrix $Q(x)$ (see \cref{ass:crcq}) has a regularity property near the solution set $\mathcal{X}^*$. Recall the index set for all active constraints $\mathcal{S}_x$ (see \cref{def:Act}), then $Q(x)$ has the following property.

\begin{lemma}\label{lem:non-singular}
Denote $\mathcal{S}\subseteq \mathcal{S}_x$, there exists constants $\epsilon_0, \eta_0>0$ such that for any $x\in \mathcal{X}$ with $\mathrm{dist}(x, \mathcal
{X}^*)< \epsilon_0$, if $Q_{\mathcal{S}}(x)$ is nonsingular then $\sigma_{\rm min}(Q_{\mathcal{S}}(x))\ge \eta_0$. 
\end{lemma}

\begin{proof}
We prove it by contradiction. Suppose the contrary, then there exists a sequence ${x}^i$  such that $\lim_{i\rightarrow \infty}\mathrm{dist}({x}^i, \mathcal{X}^*)=0$ and $\sigma_{\textrm{min}}(Q_{\mathcal{S}_i}({x}^i))\rightarrow 0$ with $Q_{\mathcal{S}_i}({x}^i)$ being nonsingular ($\mathcal{S}_i=\mathcal{S}_{x^i}$ for short). Since the choice for $\mathcal{S}_i$ is finite, there exists a sub-sequence ${x}^{i_j}$ of  $\{{x}^i\}$ such that $\mathcal{S}_{i_j}=\mathcal{S}$ for some $\mathcal{S}$. For notational simplicity, we still denote the sub-sequence as $\{{x}^i\}$ in the rest part of the proof. Since $\mathrm{dist}({x}^i, \mathcal{X}^*)\rightarrow 0$ and $\mathcal{X}^*$ is compact (by \cref{compactness}), there exists an ${x}\in \mathcal{X}^*$ such that ${x}^i\rightarrow {x}$. Since $\sigma_{\rm min}(Q)$ is a continuous function with respect to $Q$, we have $\sigma_{\rm min}(Q_{\mathcal{S}}({x}))=0$.
Then the rank of $Q_{\mathcal{S}}({x})$ is smaller than that of $Q_{\mathcal{S}}({x}^i)$. 	
On the other hand, by the continuity of $h_i$, we have $\mathcal{S}\subseteq \mathcal{S}_x$. Hence, $Q_{\mathcal{S}}({x})$ is a row submatrix of $Q_{\mathcal{S}_x}({x})$. This contradicts the CRCQ condition in \cref{ass:crcq}.
\end{proof}

The following lemma shows that the multipliers are bounded.
\begin{lemma}\label{lem:mu_bound}
For any $(r, z)\in \Omega$ and any basic set $\mathcal{S}$ of $(r, z)$, there exists a constant $U_2>0$, such that the 
pair $(y, \mu)$  satisfies $\|(y, \mu)_{\mathcal{S}}\|\le U_2$.
Moreover, since $y_i=\mu_j=0$ for $i, j\notin \mathcal{S}$, we further have $\| (y, \mu) \|\le U_2.$
\end{lemma}
\begin{proof}
Let $(y, \mu)$ be the multipliers corresponding to the basic set $\mathcal{S}$ and let $v=(y, \mu)_{\mathcal{S}}$.
By KKT conditions in \eqref{eq:KKT_xrz}, we have
$$\nabla_xf(x(y, z))+\gamma(Ax(y, z)-b)+p(x(y, z)-z)=Q_{\mathcal{S}}^T(x(y, z))v.$$
Let 
$$M_2=\max_{x, z\in \mathcal{X}}\{\|\nabla_xf(x(y, z))+\gamma(Ax(y, z)-b)+p(x(y, z)-z)\|\}.$$
By \cref{lem:non-singular}, we have $\sigma_{\rm min}(Q_{\mathcal{S}}(x(y, z))\geq \eta_0$. Therefore, $\|v\|\le M_2/\eta_0.$ Setting $U_2=M_2/\eta_0$ completes the proof.
\end{proof}
Using \cref{lem:mu_bound} and~\cref{lem:non-singular}, we can prove \cref{lem:inverse}.

\medskip
\noindent \textbf{Proof of \cref{lem:inverse}:} According to the KKT conditions in~\eqref{eq:KKT_xrz} and~\cref{def:basicset}, 
\begin{equation*}\label{Q1}
Q^T_{\mathcal{S}}(\hat{x}(r, z))v=\nabla_x\left(f(\hat{x}(r, z))+\frac{\gamma}{2}\|A\hat{x}(r, z)-b\|^2+\frac{p}{2}\|\hat{x}
(r, z)-z\|^2\right)
\end{equation*}
and
\begin{equation*}\label{Q2}
Q^T_{\mathcal{S}}(\hat{x}(r^\prime, z))v^\prime=\nabla_x\left(f(\hat{x}(r^\prime, z))+\frac{\gamma}{2}\|A\hat{x}(r^\prime, 
z)-b\|^2+\frac{p}{2}\|\hat{x}(r^\prime, z)-z\|^2\right),
\end{equation*}
where $v=(y, \mu)_{\mathcal{S}}$ and $v'=(y', \mu')_{\mathcal{S}}$.
For notational simplicity, denote
$$Q=Q_{\mathcal{S}}(\hat{x}(r, z)),\quad w=\nabla_x\left(f(\hat{x}(r, z))+\frac{\gamma}{2}\|A\hat{x}(r, z)-b\|^2+\frac{p}{2}\|\hat{x}(r, z)-z\|^2\right),$$
$$Q^\prime=Q_{\mathcal{S}}(\hat{x}(r^\prime, z)),\quad w'=\nabla_x\left(f(\hat{x}(r^\prime, z))+\frac{\gamma}{2}\|A\hat{x}(r^\prime, z)-b\|^2+\frac{p}{2}\|\hat{x}(r^\prime, z)-
z\|^2\right).$$
Then we have
\begin{equation}
\begin{aligned}\label{Q3}
&Q^T(v-v^\prime)\\
&= (Q^\prime-Q)^T v^\prime+(w-w')\\
&\le\sqrt{\ell} L_{h}\|v^\prime\|\|\hat{x}(r, z)-\hat{x}(r^\prime, z)\|+(L_f+\gamma\sigma_{\max}^2(A)+p)\|\hat{x}(r, z)-\hat
{x}(r^\prime, z)\|,
\end{aligned}
\end{equation}
where the last inequality is because $\nabla_xh_i(\cdot)$ is $L_h$-Lipschitz-continuous  and $\nabla_x(f
(x)+\frac{\gamma}{2}\|Ax-b\|^2+\frac{p}{2}\|x-z\|^2)$ is $(L_f+\gamma\sigma_{\rm{max}}^2(A)+p)$-Lipschitz continuous.
On the other hand, by \cref{lem:mu_bound}, we have $\|v'\|<U_2.$ Since $\mathcal{S}$ is a basic set, $Q$ is of full row rank. Hence, $Q^T$ is of full column rank and $\sigma_{\rm min}(Q^T)>\eta_0$.
Then according to \cref{lem:non-singular}, we have $\|Q^T(v-v')\|\ge \eta_0\|v-v'\|$. Therefore, by \eqref{Q3}, we have
\begin{eqnarray*}
\mathrm{dist}(\mathcal{Y}^*(r), \mathcal{Y}^*(r'))&\le& \|v-v'\|\\
&\le&(\sqrt{\ell}L_hU_2/\eta_0+(L_f+\gamma\sigma_{\max}^2(A)+p)/\eta_0)\|\hat{x}(r, z)-\hat{x}(r', z)\|.
\end{eqnarray*}
Setting $\kappa=(\sqrt{\ell}L_hU_2/\eta_0+(L_f+\gamma\sigma_{\max}^2(A)+p)/\eta_0)$ completes the proof.

\subsection{Proof of the \cref{dual-pieces1}}\label{app:propdec}
Two technical lemmas are introduced. The first lemma shows the continuity of $\hat{x}(r, z)$ in $r$.
\begin{lemma}\label{continuity2}
$\hat{x}(r, z)$ is continuous of $r$  for $(r, z)\in \Omega$.
\end{lemma}
\begin{proof}
We prove it by contradiction. Let $(r^i, z), (\bar{r}, z)\in \Omega$ and the sequence $\{r^i\}$ converges to $\bar{r}$. 
Suppose the contrary. Then $\hat{x}(r^i, z)$ does not converge to $\hat{x}(\bar{r}, z)$. 
Since the choices of basic sets are finite, there must exist a subindex set $\{i_j\}$ such that all vectors in the 
subsequence $r^{i_j}$ share a common basic set $\mathcal{S}$ with $\bar{r}$.
Therefore by the definition of the basic set (\cref{def:basicset}) we have ($\hat{x}^{i_j}=\hat{x}(r^{i_j}, z)$)
\begin{equation}
\begin{aligned}\label{KKTforr-basic}
\nabla_x f(\hat{x}^{i_j})+A^T\bar{y}^{i_j}+\gamma A^T(A\hat{x}^{i_j}-b)&+p(x^{i_j}-z)+J^T(\hat{x}^{i_j})\mu^{i_j}=0,\\
A\hat{x}^{i_j}-b&=r^{i_j},\\
h_i(\hat{x}^{i_j})=0,\ \mu_i^{i_j}&\ge0, i\in \mathcal{S},\\
h_i(\hat{x}^{i_j})\le 0,\ \mu_i^{i_j}&=0, i\notin \mathcal{S},\\
\bar{y}_i^{i_j}&= 0, i\notin \mathcal{S}.
\end{aligned}
\end{equation}
By \cref{lem:mu_bound}, we have $\|(\bar{y}^{i_j}, \mu^{i_j})\|\le U_2$ and $\{(\hat{x}(r^{i_j}, z), \bar
{y}^{i_j}, \mu^{i_j})\}$ is bounded. Then, there exists  a limit point $({x},\bar{y}, \bar{\mu}) $ of $\{(\hat{x}(r^
{i_j}, z), \bar{y}^{i_j}, \mu^{i_j})\}$ such that 
$$\{(\hat{x}(r^{i_j}, z), \bar{y}^{i_j}, \mu^{i_j})\}\rightarrow ({x}, \bar{y}, \bar{\mu}).$$
Taking limit with respect to $j$ of \eqref{KKTforr-basic}, we have
\begin{equation}
\begin{aligned}
\nabla_x f({x})+A^T\bar{y}+\gamma A^T(A{x}-b)+p({x}-z)+Q_{\mathcal{S}}^T({x})\bar{\mu}&=0,\\
A{x}-b&=\bar{r},\\
h_i({x})=0, \bar{\mu}_i&\ge 0, i\in \mathcal{S},\\
h_i({x})\le 0, \bar{\mu}_i&=0, i\notin \mathcal{S}\\
\bar{y}_i&=0, i\notin \mathcal{S}.
\end{aligned}
\end{equation}
Hence, ${x}=\hat{x}(\bar{r}, z)$ which is a contradiction. This completes the proof.
\end{proof}
For notational simplicity, in the remaining,  we fix some $z$ satisfying $\|z-{x}(z)\|\le \delta$ and say that $\mathcal
{S}$ is a basic set of $r$ if $\mathcal{S}$ is a basic set of $(r, z)$. The following lemma shows that the basic sets are preserved when taking limits.
\begin{lemma}\label{properties2}
Let $\{r^i\}$ be a sequence with $\|r^i\|\leq \Delta$ and $\mathcal{S}\subseteq [m+\ell]$. Suppose $r^i\rightarrow r$ with $\|r\|\leq\Delta$, then
\begin{enumerate}
\item  For any $r^\prime$ with $\|r^\prime\|\leq\Delta$, if all $r^i$ share $\mathcal{S}$ with $r^\prime$, then $r$ shares $\mathcal{S}$ with $r^\prime$.
\item If all $r^i$ share a common basic set $\mathcal{S}$, then $\mathcal{S}$ is also a basic set of $r$.
\item The set $\mathcal{R}(\mathcal{S})=\{r \mid \|r \|\leq\Delta,\ \mathcal{S} \text{ is a 
basic set of }$r$ \}$ is closed and compact.

\end{enumerate}
\end{lemma}
\begin{proof}
For any $r^i$ with $\|r^i\|\leq \Delta$, we have the following KKT conditions corresponding to $\mathcal{S}$ ($\hat{x}^{i}=\hat{x}(r^{i}, z)$):
\begin{equation}
\begin{aligned}
\nabla_x f(\hat{x}^{i})+A^T\bar{y}^i+\gamma A^T(A\hat{x}^{i}-b)+p(\hat{x}^{i}-z)+&Q_{\mathcal{S}}^T(\hat{x}^{i})\bar{\mu}^i=0,\\
A\hat{x}^{i}-b&=r^i,\\
h_j(\hat{x}^{i})&=0,\ \bar{\mu}_j^i\ge 0, j\in \mathcal{S},\\
h_j(\hat{x}^{i})&\le 0,\ \bar{\mu}_j^i=0, \ \bar{y}_j^i=0, j\notin \mathcal{S}.
\end{aligned}
\end{equation}

By the continuity  property of $\hat{x}(r, z)$ in \cref{continuity2}, we have $\lim_{i\rightarrow \infty}\hat{x}(r^i, 
z)=\hat{x}(r, z).$
Also, by  \cref{lem:mu_bound}, the sequences $\{\bar{y}^i\}$ and $\{\bar{\mu}^i\}$ are bounded. Hence each of these 
two sequences has at least one limit point. Passing to a sub-sequence if necessary, we assume that 
$\bar{y}^i\rightarrow \bar{y}, \bar{\mu}^i\rightarrow \bar{\mu}.$ Taking limit to the above KKT conditions when $i\rightarrow  \infty$, we have
\begin{equation}
\begin{aligned}
\nabla_x f(\hat{x}(r, z))+A^T\bar{y}+\gamma A^T(A\hat{x}(r, z)-b)&+p(\hat{x}(r, z)-z)+Q_{\mathcal{S}}^T(\hat{x}(r, z))\bar
{\mu}=0,\\
A\hat{x}(r, z)-b&=r,\\
h_j(\hat{x}(r, z))&=0,\ \bar{\mu}_j\ge 0, j\in \mathcal{S},\\
h_j(\hat{x}(r, z))&\le 0,\ \bar{\mu}_j=0,\ \bar{y}_j=0, j\notin \mathcal{S}.
\end{aligned}
\end{equation}
Moreover, by \cref{lem:non-singular}, $\sigma_{\rm min}(Q_{\mathcal{S}}(\hat{x}(r^i, z))\ge \eta_0$ and hence 
$\sigma_{\rm min}(Q_{\mathcal{S}}(\hat{x}(r, z)))\ge \eta_0$.
This implies that $Q_{\mathcal{S}}(\hat{x}(r, z))$ is of full row rank.
Hence, $\mathcal{S}$ is a basic set of $r$. This completes the proof of the first part. The second part and the third part 
are direct corollaries of the first part.
\end{proof}

\medskip
\noindent \textbf{Proof of \cref{dual-pieces1}:} We prove it by constructing $\lambda_i$ recursively. Suppose we already have $0=\lambda_0<\lambda_1<\cdots <\lambda_{i}<1$, 
we try to find a suitable $\lambda_{i+1}$ such that $\lambda_{i}<\lambda_{i+1}\leq 1$.
Notice that the choice of basic sets for a given $\tilde{r}$ is finite, i.e., a basic set must be a subset of $[m+\ell]$. Denote $\mathcal{B}_1, \mathcal{B}_2, \cdots, \mathcal{B}_{2^{m+\ell}}$ as all the subsets of the index set $[m+\ell]$ and define index set $\mathcal{J}_i$ as  
$$ \mathcal{J}_i=\{j\in[2^{m+\ell}]\mid \mathcal{B}_j\text{ is a basic set of } \lambda_i\tilde{r}\}\subseteq [2^{m+\ell}].$$
According to \cref{properties2}, $\mathcal{R}(\mathcal{B}_j)$ is compact for any $j\in [2^{m+\ell}]$.
Then $\cup_{j\in \mathcal{J}_i}\mathcal{R}(\mathcal{B}_j)$ is a compact set since $\mathcal{J}_i$ is a finite set.
Denote set $\mathcal{C}_i$ as below $$\mathcal{C}_i= \cup_{j\in \mathcal{J}_i}\mathcal{R}(\mathcal{B}_j)\cap \{\lambda\tilde{r}\mid \lambda\in [\lambda_i, 1]\},$$ 
which represents the set of all vectors in the segment connecting $\lambda_i\tilde{r}$ and $\tilde{r}$ that have at least one common basic set with $\lambda_i\tilde{r}$.
Consider a continuous function $\psi(r)=\|r\|/\|\tilde{r}\|$, which can take a exact maximum in the compact set $\mathcal{C}_i$. Then we define $\lambda_{i+1}$ to be
$$\lambda_{i+1}=\max_{r\in \mathcal{C}_i}\psi(r).$$
Then $\lambda_{i+1}$ is the largest $\lambda\in [\lambda_{i}, 1]$ such that $\lambda\tilde{r}$ shares a common basic set with $\lambda_{i}\tilde{r}$.

Next, we prove that this construction will not be stuck, i.e., $\lambda_{i+1}>\lambda_i$ if $\lambda_i<1$.
Let $\{\xi_k\}\subseteq (\lambda_i, 1]$ be a sequence converging to  $\lambda_i$. Since the choice for basic sets is 
finite, there exists a subsequence $\{\xi_{k_j}\}$ of $\{\xi_k\}$ such that all $\xi_{k_j}\tilde{r}$ share a common basic 
set $\mathcal{S}$.
Since $\lim_{j\rightarrow \infty}\xi_{k_j}\tilde{r}= \lambda_i\tilde{r}$, by part 2 in \cref{properties2}, we know 
that $\mathcal{S}$ is also a basic set of $\lambda_i\tilde{r}$. This implies that there exists at least one $\lambda\in 
(\lambda_i, 1 ]$ such that $\lambda\tilde{r}$ shares a common basic set with $\lambda_i\tilde{r}$. Hence, $\lambda_{i+1}>
\lambda_i$ if $\lambda_i<1$.

Finally, we prove that the sequence $\{\lambda_i\}$ constructed above is finite, i.e., there exists some $R>0$ such that $\lambda_R=1$. 
We prove it by contradiction. Suppose that $0=\lambda_0<\lambda_1<\cdots<\lambda_k<\cdots$ is infinite. 
Since the choice for basic sets is finite,  there must exist  ${i_1}<{i_2}<{i_3}$ such that $\lambda_{i_1}\tilde{r}$, $
\lambda_{i_2}\tilde{r}$ and $\lambda_{i_3}\tilde{r}$ share a common basic set $\mathcal{S}$. 
Then $i_3>i_1+1$. It  contradicts the definition that $\lambda_{i_1+1}$ is the largest $\lambda\in [\lambda_{i_1}, 1]$ such 
that $\lambda\tilde{r}$ shares a common basic set with $\lambda_{i_1}\tilde{r}$. Therefore, the sequence $\{\lambda_i\}$ is 
finite. This completes the proof.

\end{document}